\documentclass{amsart}

\usepackage{curves}
\usepackage{epic}
\usepackage{eepic}
\usepackage{amsmath}
\usepackage{amscd}
\usepackage{amsfonts}
\usepackage{amssymb}
\usepackage{latexsym}
\usepackage{pifont}
\usepackage{array}
\usepackage{oldgerm}

\newtheorem{thm}{Theorem}[section]
\newtheorem{pro}[thm]{Proposition}
\newtheorem{lem}[thm]{Lemma}
\newtheorem{cor}[thm]{Corollary}
\newtheorem{thm&def}[thm]{Theorem \& Definition}
\newtheorem{lem&def}[thm]{Lemma \& Definition}

\theoremstyle{definition}
\newtheorem{defi}[thm]{Definition}                   

\newtheorem{exa}[thm]{Example}

\newtheorem{rmk}[thm]{Remark}

\newcommand{\ncm}{\newcommand}

\ncm{\decorenum}[1]{

\begin{enumerate}}
\ncm{\decorenumend}{
\end{enumerate}

}

\ncm{\decorenumprime}[1]{

\begin{enumerate}}
\ncm{\decorenumprimeend}{
\end{enumerate}

}

\ncm{\Cat}{\mathsf{Cat}}
\ncm{\CAT}{\mathsf{CAT}}
\ncm{\ob}{\operatorname{ob}}
\ncm{\Nat}{\operatorname{Nat}}
\ncm{\Set}{\mathsf{Set}}
\ncm{\Ab}{\mathsf{Ab}}
\ncm{\sfM}{\mathsf{M}}
\ncm{\Add}{\mathsf{Add}}
\ncm{\Mnd}{\mathsf{Mnd}}
\ncm{\Fun}{\mathsf{Fun}}
\ncm{\MonCat}{\mathsf{MonCat}}
\ncm{\MonFun}{\mathsf{MonFun}}
\ncm{\Elt}{\mathsf{Elt}\,}
\ncm{\Sub}{\mathsf{Sub}\,}
\ncm{\Flat}{\mathsf{Flat}}
\ncm{\add}{\text{\Large\textsl{a}}}
\ncm{\proj}{\mathsf{proj}}
\ncm{\Col}{\mathsf{Col}\,}
\ncm{\fgmod}[1]{\mathsf{mod}\text{-}#1}
\ncm{\Split}{\mathsf{split}}
\ncm{\Cl}{\mathsf{Cl}}
\ncm{\ord}{\mathsf{ord}}
\ncm{\RightMonCat}{\textsf{r-MonCat}}
\ncm{\LeftMonCat}{\textsf{l-MonCat}}
\ncm{\RightOpmonCat}{\textsf{r-OpmonCat}}
\ncm{\LeftOpmonCat}{\textsf{l-OpmonCat}}
\ncm{\LaxCmd}{\mathsf{Lax}\text{-}\mathsf{Cmd}}

\ncm{\A}{\mathcal{A}}
\ncm{\B}{\mathcal{B}}
\ncm{\C}{\mathcal{C}}
\ncm{\D}{\mathcal{D}}
\ncm{\E}{\mathcal{E}}
\ncm{\F}{\mathcal{F}}
\ncm{\G}{\mathcal{G}}
\ncm{\Ha}{\mathcal{H}}
\ncm{\I}{\mathcal{I}}
\ncm{\J}{\mathcal{J}}
\ncm{\K}{\mathcal{K}}
\ncm{\Ll}{\mathcal{L}}
\ncm{\M}{\mathcal{M}}
\ncm{\N}{\mathcal{N}}
\ncm{\Ou}{\mathcal{O}}
\ncm{\Pee}{\mathcal{P}}
\ncm{\R}{\mathcal{R}}
\ncm{\X}{\mathcal{X}}
\ncm{\V}{\mathcal{V}}
\ncm{\U}{\mathcal{U}}
\ncm{\T}{\mathcal{T}}

\ncm{\dom}{\operatorname{dom}}
\ncm{\cod}{\operatorname{cod}}
\ncm{\End}{\operatorname{End}}
\ncm{\Aut}{\operatorname{Aut}}
\ncm{\Hom}{\operatorname{Hom}}
\ncm{\kernel}{\operatorname{ker}}
\ncm{\Ker}{\operatorname{Ker}}
\ncm{\coker}{\operatorname{coker}}
\ncm{\Coker}{\operatorname{Coker}}
\ncm{\im}{\operatorname{im}}
\ncm{\Img}{\operatorname{Im}}
\ncm{\coim}{\operatorname{coim}}
\ncm{\id}{\operatorname{id}}
\ncm{\Center}{\operatorname{Center}}
\ncm{\colim}{\operatorname{colim}}
\ncm{\Colim}[1]{\underset{#1}{\operatorname{colim}}}
\ncm{\Lan}{\operatorname{Lan}}
\ncm{\Cone}{\operatorname{Cone}}
\ncm{\ev}{\operatorname{ev}}
\ncm{\cf}{\operatorname{cf}}
\ncm{\hgt}{\operatorname{ht}}

\ncm{\ci}{\circ}
\ncm{\bu}{\bullet}
\ncm{\bo}{\,\Box\,}
\ncm{\ot}{\otimes}
\ncm{\x}{\times}
\ncm{\smp}{\ast}
\ncm{\smpq}{\smp_q}
\ncm{\smpz}{\smp_z}
\ncm{\rmpr}{\,\triangledown\,}
\ncm{\rmpl}{\vartriangle}
\ncm{\lmpr}{\blacktriangledown}
\ncm{\lmpl}{\blacktriangle}
\ncm{\hor}[1]{\Diamond}
\ncm{\ex}[1]{\underset{\ssst #1}{\times}}
\ncm{\am}[1]{\underset{\scriptscriptstyle #1}{\ot}}
\ncm{\amo}[1]{\underset{\scriptstyle #1}{\ot}}
\ncm{\mash}{\Pisymbol{psy}{35}}
\ncm{\mashed}[1]{\underset{\scriptscriptstyle #1}{\Pisymbol{psy}{35}}}
\ncm{\cross}[1]{\underset{\scriptstyle #1}{\rtimes}}
\ncm{\rarr}[1]{\stackrel{#1}{\longrightarrow}}
\ncm{\larr}[1]{\stackrel{#1}{\longleftarrow}}
\ncm{\mapsot}{\leftarrow\!\!\!\raisebox{1pt}{$\scriptscriptstyle |$}}
\ncm{\oR}{\am{R}}
\ncm{\oS}{\am{S}}
\ncm{\cop}{\Delta}
\ncm{\oneT}{^{(1)}}
\ncm{\twoT}{^{(2)}}
\ncm{\threeT}{^{(3)}}
\ncm{\oneB}{_{(1)}}
\ncm{\twoB}{_{(2)}}
\ncm{\threeB}{_{(3)}}
\ncm{\eps}{\varepsilon}
\ncm{\bra}{\langle}
\ncm{\ket}{\rangle}

\ncm{\asso}{\mathbf{a}}
\ncm{\luni}{\mathbf{l}}
\ncm{\runi}{\mathbf{r}}
\ncm{\tet}{t}
\ncm{\iso}{\stackrel{\sim}{\rightarrow}}
\ncm{\iiso}{\rarr{\sim}}
\ncm{\ract}{\triangleleft}
\ncm{\lact}{\triangleright}
\ncm{\under}{\mbox{\,\rm\_}\,}
\ncm{\adj}{\dashv}
\ncm{\adjoint}{\dashv}
\ncm{\into}{\hookrightarrow}

\ncm{\can}{\mathrm{can}}
\ncm{\fgp}{\mathrm{fgp}}
\ncm{\op}{\mathrm{op}}
\ncm{\coop}{\mathrm{coop}}
\ncm{\rev}{\mathrm{rev}}
\ncm{\sst}{\scriptstyle}
\ncm{\ssst}{\scriptscriptstyle}
\ncm{\sH}{{\ssst H}}
\ncm{\sR}{{\ssst R}}

\ncm{\eqby}[1]{\stackrel{(\ref{#1})}{=}}
\ncm{\lef}{{\ssst <}}
\ncm{\righ}{{\ssst >}}

\ncm{\one}{\mathbf{1}}
\ncm{\NN}{\mathbb{N}}
\ncm{\ZZ}{\mathbb{Z}}
\ncm{\QQ}{\mathbb{Q}}
\ncm{\GG}{\mathbf{G}}
\ncm{\FF}{\mathbb{F}}
\ncm{\TT}{\mathsf{T}}
\ncm{\Q}{\mathsf{Q}}
\ncm{\bfT}{\mathbf{T}}
\ncm{\bfQ}{\mathbf{Q}}
\ncm{\ch}{\mathbf{H}}

\ncm{\pisharp}{\Pisymbol{psy}{35}}

\ncm{\Pre}{\hat\C}
\ncm{\She}{\mathsf{Sh}}
\ncm{\Sig}{{\Sigma}}
\ncm{\icog}{J}

\ncm{\eL}{\mathbf{L}}
\ncm{\eR}{\mathbf{R}}

\ncm{\parallelpair}{
\parbox{43pt}{
\begin{picture}(43,8)
\put(3,6){\vector(1,0){37}}
\put(3,2){\vector(1,0){37}}
\end{picture}
}}
\ncm{\pair}[2]{\overset{#1}{\underset{#2}{\parallelpair}}}

\ncm{\longparallelpair}{
\parbox{63pt}{
\begin{picture}(63,8)
\put(3,6){\vector(1,0){57}}
\put(3,2){\vector(1,0){57}}
\end{picture}
}}
\ncm{\longpair}[2]{\overset{#1}{\underset{#2}{\longparallelpair}}}

\ncm{\longerparallelpair}{
\parbox{83pt}{
\begin{picture}(83,8)
\put(3,6){\vector(1,0){77}}
\put(3,2){\vector(1,0){77}}
\end{picture}
}}
\ncm{\longerpair}[2]{\overset{#1}{\underset{#2}{\longerparallelpair}}}

\ncm{\longrightarrowtail}{
\parbox{40pt}{
\begin{picture}(40,8)
\put(6,4){\line(-1,1){1.5}}
\put(6,4){\line(-1,-1){1.5}}
\put(6,4){\vector(1,0){31}}
\end{picture}
}}

\ncm{\longerrightarrowtail}{
\parbox{60pt}{
\begin{picture}(60,8)
\put(6,4){\line(-1,1){1.5}}
\put(6,4){\line(-1,-1){1.5}}
\put(6,4){\vector(1,0){51}}
\end{picture}
}}

\ncm{\longrarr}[1]{
\overset{#1}{
\parbox{40pt}{
\begin{picture}(40,8)
\put(3,4){\vector(1,0){34}}
\end{picture}
}}}

\ncm{\longerrarr}[1]{
\overset{#1}{
\parbox{70pt}{
\begin{picture}(70,8)
\put(3,4){\vector(1,0){64}}
\end{picture}
}}}

\ncm{\longlarr}[1]{\overset{#1}{\longleftarrow}}

\ncm{\antiparallelpair}{
\parbox{23pt}{
\begin{picture}(23,4)
\put(3,3){\vector(1,0){17}}
\put(20,1){\vector(-1,0){17}}
\end{picture}
}}

\ncm{\invantiparallelpair}{
\parbox{23pt}{
\begin{picture}(23,4)
\put(3,1){\vector(1,0){17}}
\put(20,3){\vector(-1,0){17}}
\end{picture}
}}

\ncm{\dualpair}[2]{\overset{#1}{\underset{#2}{\antiparallelpair}}}

\ncm{\invdualpair}[2]{\overset{#1}{\underset{#2}{\invantiparallelpair}}}

\ncm{\coantiparallelpair}{
\parbox{23pt}{
\begin{picture}(23,4)
\put(3,1){\vector(1,0){17}}
\put(20,4){\vector(-1,0){17}}
\end{picture}
}}

\ncm{\codualpair}[2]{\overset{#1}{\underset{#2}{\coantiparallelpair}}}

\ncm{\binarydirectsum}[7]{#1\codualpair{#2}{#3}#4\dualpair{#5}{#6}#7}

\ncm{\equalizer}[1]{\overset{#1}{\longrightarrowtail}}

\ncm{\epi}[1]{\overset{#1}{\twoheadrightarrow}}

\ncm{\coequalizer}[1]{
\overset{#1}{
\parbox{40pt}{
\begin{picture}(40,8)
\put(2,4){\vector(1,0){32}} \put(37,4){\vector(1,0){0}}
\end{picture}
}}}

\ncm{\mono}[1]{\overset{#1}{\rightarrowtail}}

\begin{document}

\title{Skew-monoidal categories and bialgebroids}
\author{Korn\'el Szlach\'anyi} 
\date{}
\thanks{Supported by the Hungarian Scientific Research Fund, OTKA 68195.\newline
\indent Appeared in Adv. Math. 231 (2012) 1694-1730}
\begin{abstract}Skew-monoidal categories arise when the associator and the left and right units of a monoidal category are,
in a specific way, not invertible. We prove that the closed skew-monoidal structures on the category of right $R$-modules are
precisely the right bialgebroids over the ring $R$. These skew-monoidal structures induce quotient skew-monoidal structures on the category of $R$-$R$-bimodules and this leads to the following generalization: Opmonoidal monads on a monoidal category
correspond to skew-monoidal structures with the same unit object which are compatible with
the ordinary monoidal structure by means of a natural distributive law. 
Pursuing a Theorem of Day and Street we also discuss monoidal lax comonads to describe the comodule categories of bialgebroids
beyond the flat case.
\end{abstract}
\address{Wigner Research Centre for Physics of the Hungarian Academy of Sciences,\newline
\noindent 1525 Budapest, P.O.Box 49, Hungary}
\email{szlachanyi.kornel@wigner.mta.hu}

\maketitle

\section{Introduction}

Bialgebroids \cite{Takeuchi, Lu, Xu, KSz} are generalizations of bialgebras to non-commutative base ring. 
By replacing the commutative base ring $k$ of a bialgebra with a non-commutative ring $R$ the symmetric role of the monoid
and comonoid structure is lost: A bialgebroid $H$ over $R$ is a comonoid $H\rarr{\Delta}H\oR H$ in the category $_R\Ab_R$ of $R$-bimodules but a monoid $H\ot_{R^e}H\rarr{m}H$ in the category of 
$R^e:=R^\op\ot R$-bimodules. The compatibility condition between the $R^e$-ring and the $R$-coring structure is too complicated to witness about something fundamental which may motivate to search for other generalizations of bialgebras \cite{Mes-Wis}.
However, if we look at the functor $\under\ot_{R^e}H$ on the monoidal category $\Ab_{R^e}=\,_R\Ab_R$ instead of the object
$H\in \,_{R^e}\Ab_{R^e}$ itself, the condition becomes amazingly simple.
As it was observed in \cite{Sz: EM} a bimodule $H$ is a bialgebroid precisely if $\under\ot_{R^e}H$ is an opmonoidal monad \cite{Moerdijk, McCrudden}. 

The language of monads tells us that the modules over the bialgebroid $H$ have to be the objects of the Eilenberg-Moore category
of the monad $\under\ot_{R^e}H$. Opmonoidality is then precisely the structure that makes the category of modules monoidal and
the Eilenberg-Moore forgetful functor strict monoidal.
This gives nothing new with respect to the `classical' algebraic formalism:
The Eilenberg-Moore category is the category of $H$-modules ($H$ as an $R^e$-ring). But what are the comodules of an opmonoidal monad? The monadic language gives no hint. Classically one knows that there is the category of comodules over the $R$-coring $H$ 
and several authors argued \cite{Phung, Bohm, Ba-Sz} that this category becomes monoidal with a strict monoidal forgetful functor
to $_R\Ab_R$. This comodule category, however, is not the Eilenberg-Moore category of a monoidal comonad (unless $H$ is flat
as left $R$-module) which is a further asymmetry between modules and comodules of bialgebroids. Instead of monoidal comonad
there is a lax monoidal structure given by Takeuchi's $\x_R$-product with respect to which bialgebroids can be seen as comonoids
\cite{Day-Street} and therefore have comodules in a natural way. 

In this paper, we propose to consider a fragment of the structure of bialgebroids which lets their modules and comodules seen symmetrically or, better to say, dually. This fragment, called a skew-monoidal category, has left and right versions just like bialgebroids have \cite{KSz}. A right-monoidal category consists of a category $\M$, a functor $\M\x\M\rarr{\smp}\M$, 
an object $R\in\M$ and comparison natural transformations
\[
L\smp(M\smp N)\rarr{\gamma}(L\smp M)\smp N,\qquad M\rarr{\eta}R\smp M,\qquad M\smp R\rarr{\eps}M
\]
satisfying the usual pentagon and triangle equations of a monoidal category without assuming, however, invertibility of either
$\gamma$, $\eta$ or $\eps$. In left-monoidal categories all comparisons go in the opposite way and the names $\eta$ and $\eps$ are interchanged.
For a right bialgebroid $H$ over $R$ the category $\M$ is the category $\Ab_R$ of right $R$-modules, $R$ is the
regular right $R$-module, $\eps$ and $\eta$ are essentially the counit and the source map of $H$, respectively, while the skew-associator $\gamma$ is the Galois map or canonical map $H\oR H\to H\bar\oR H$ built of the multiplication and comultiplication of $H$. What is not so simple to explain is the skew-monoidal product $\smp$.

The advantage of looking at the skew-monoidal category $\M$ instead of the bialgebroid $H$ is that it encodes all information on the 
categories of right $H$-modules and of right $H$-comodules as simply as the Eilenberg-Moore categories of the canonical monad $T=R\smp\under$ and of the canonical comonad $Q=\under\smp R$ on $\M$. The disadvantage is that their monoidal structure is not seen. It is hidden in the properties of the category $\M$ together with all asymmetries between modules and comodules encoded in exactness properties of $\M$ and $\smp$. 

Generalizations of monoidal categories or bicategories by relaxing invertibility of the comparison cells are not unknown in the literature. Burroni's \textit{pseudocategory} \cite{Burroni} has comparison cells $(L\smp M)\smp N\to L\smp (M\smp N)$,
$M\to R\smp M$, $M\to M\smp R$ and Grandis'  \textit{d-lax 2-category} \cite{Grandis} has 
$L\smp(M\smp N)\to(L\smp M)\smp N$, $R\smp M\to M$ and $M\to M\smp R$ therefore they are neither the left- nor the right-monoidal structures of the present paper. Blute, Cockett and Seely introduced the notion of \textit{context category} \cite{Blute-Cockett-Seely}
which contains, as part of the structure, precisely what we call right-monoidal comparison cells and the 5 axioms of a right-monoidal category can also be found among their axioms. Lax monoidal categories \cite{Leinster} provide another 'unbiased' way to
generalize monoidal categories which also have non-invertible comparison cells but 
no associator in the ordinary `biased' sense.
Much closer in spirit to our approach is the \textit{2-monoidal} and \textit{duoidal} \textit{categories} 
 \cite{Agu-Mah, Booker-Street} of Aguiar and Mahajan in spite of that they use two ordinary monoidal structures instead of a
`skew' one. For example the tensor square $H=R\smp R$ of the skew-monoidal unit, which is both a $T$-algebra and a $Q$-coalgebra, is reminiscent to a bimonoid in a 2-monoidal category although the precise connection is not clear.
A direct predecessor of our skew-monoidal product 
is the non-unital monoidal product $\smp$ Ross Street constructs in \cite{Street: fusion}
on a braided monoidal category equipped with a \textit{tricocycloid} $H\ot H\iso H\ot H$. Our $\gamma_{R,R,R}$ corresponds to a non-invertible tricocycloid on the object $H\in \,_R\Ab_R$ in a situation where no braiding is present.

The main result of this paper is the following characterization of
bialgebroids (Theorem \ref{thm: bgd}): The closed right-monoidal structures on $\Ab_R$ with skew-monoidal unit $R$ are precisely the right bialgebroids over $R$. Similar statement holds for left-monoidal structures on $_R\Ab$ and left bialgebroids.
The proof of this Theorem has four ingredients: 1. By left closedness of $\smp$ and by the Eilenberg-Watts Theorem there is
a natural isomorphism $M\oR TN\iso M\smp N$. 2. Right exactness of $T$ leads to a lifting of $\smp$ to a skew-monoidal product $\smpq$ on $_R\Ab_R$ which admits an isomorphism $w_{M,N}:M\oR T_qN\iso M\smpq N$ in terms of the canonical monad $T_q$
of the $\smpq$-structure. 3. The $w_{M,N}$ satisfies two coherence conditions in the form of a heptagon and a tetragon equation
which turns out to be equivalent, by our Representability Theorem (Theorem \ref{thm: representability}), to that $T$ is opmonoidal, hence a bimonad on $_R\Ab_R$. 4. Finally, by right closedness of $\smp$ this bimonad is left adjoint hence the bimonad of a bialgebroid by a Theorem of \cite{Sz: EM}. 

The Representability Theorem is valid for any category equipped with two monoidal
structures, an ordinary one $\ot$ and a skew one $\smp$, and says that $\smp$ can be expressed as $M\smp N\cong M\ot TN$
with a bimonad $T$ precisely if the two monoidal structures are related by a \textit{tetrahedral isomorphism}
$L\ot(M\smp N)\to (L\ot M)\smp N$. The skew-monoidal structures on a monoidal category that can be expressed by a bimonad as above are called \textit{representable}. This notion was inspired by the fusion operator formalism of \cite{BLV} since a 
\textit{fusion operator} 
$T(M\ot TN)\to TM\ot TN$ is the essential part of a skew-associator $\gamma_{L,M,N}$. As a matter of fact, for a bimonad $T$ the expression $M\smp N:=M\ot TN$ always defines a skew-monoidal product (Proposition \ref{pro: O-ind}).

Although the Representability Theorem can be dualized and skew-monoidal structures can be constructed from monoidal comonads this Corepresentability Theorem is not applicable to the monoidal (lax) comonad of a bialgebroid because of the different
exactness properties we encounter. It could be applicable, however, to quantum categories \cite{Day-Street} or to bicoalgebroids
\cite{Brz-Mil,Balint}. In order to complete the picture with the comodules of bialgebroids
we use a lax version of the notion of comonad in Section \ref{s: lax Q}, called \textit{cohypomonad} in \cite{Day-Panchadcharam-Street}, and show in Theorem \ref{thm: bfQ monoidal} that at least in case of the skew-monoidal category of a bialgebroid this lax comonad is monoidal. 
These results are not really new but a reformulation in a minimalistic language of what has been called in \cite{Day-Street} a
comonoid in a lax monoidal category provided by the iterated Takeuchi product.

\section{Skew-monoidal categories}

\begin{defi}
A right-monoidal category $\bra\M,\smp,R,\gamma,\eta,\eps\ket$ consists of a category $\M$, a functor $\under\smp\under:\M\x\M\to\M$, an object $R$ of $\M$ and natural transformations
\begin{align*}
\gamma_{L,M,N}&:L\smp(M\smp N)\to(L\smp M)\smp N\\
\eta_M&:M\to R\smp M\\
\eps_M&:M\smp R\to M
\end{align*}
subject to the following axioms: For all objects $K$, $L$, $M$, $N$ 
\begin{align}
\label{SMC1}
(\gamma_{K,L,M}\smp N)\ci\gamma_{K,L\smp M,N}\ci(K\smp\gamma_{L,M,N})&=\gamma_{K\smp L,M,N}\ci\gamma_{K,L,M\smp N}\\
\label{SMC2}
\gamma_{R,M,N}\ci\eta_{M\smp N}&=\eta_M\smp N\\
\label{SMC3}
\eps_{M\smp N}\ci\gamma_{M,N,R}&=M\smp\eps_N\\
\label{SMC4}
(\eps_M\smp N)\ci\gamma_{M,R,N}\ci(M\smp\eta_N)&=M\smp N\\
\label{SMC5}
\eps_R\ci\eta_R&=R
\end{align}
\end{defi}

If we replace $\M$ with $\M^{\op,\rev}$, the category with opposite composition and with right-monoidal product of reversed order,
we obtain again a right-monoidal category, with roles of $\eta$ and $\eps$ interchanged. But replacing $\M$ with either $\M^\op$ or $\M^\rev$ what we obtain is different from the above structure. We call it a left-monoidal category.

If $\gamma$, $\eta$, $\eps$ are isomorphisms we recover the notion of a monoidal category with somewhat strange names for the associator and left and right units. 

\begin{defi} \label{def: skew-mon func}
If $\M$ and $\N$ are right-monoidal categories (with structures denoted by $\smp$, $R$, $\gamma$, $\eta$, $\eps$ in both cases)
then a right-monoidal functor $\M\to\N$ is a triple $\bra F,F_2,F_0\ket$ where $F$ is a functor $\M\to\N$ of the underlying categories, $F_0$ is an arrow $R\to FR$ and $F_2$ is a natural transformation $F_{X,Y}:FX\smp FY\to F(X\smp Y)$ satisfying
\begin{align}
F\gamma_{X,Y,Z}\ci F_{X,Y\smp Z}\ci (FX\smp F_{Y,Z})&=F_{X\smp Y,Z}\ci(F_{X,Y}\smp FZ)\ci\gamma_{FX,FY,FZ}
\label{smf1}\\
F_{R,X}\ci(F_0\smp FX)\ci\eta_{FX}&=F\eta_X
\label{smf2}\\
F\eps_X\ci F_{X,R}\ci(FX\smp F_0)&=\eps_{FX}
\label{smf3}
\end{align}
for all $X,Y,Z\in\M$. Left-monoidal functors are similar functors between left-monoidal categories. They together will be referred
to as skew-monoidal functors.

A skew-opmonoidal functor $\M\to\N$ is a triple $\bra F,F^2,F^0\ket$ where $F$ is a functor $\M\to\N$, $F^0$ is an arrow 
$FR\to R$ and $F^2$ is a natural transformation $F^{X,Y}:F(X\smp Y)\to FX\smp FY$ such that $F$, $F_0:=F^0$ and
$F_{X,Y}:=F^{Y,X}$ define a skew-monoidal functor $\M^{\op,\rev}\to\N^{\op,\rev}$.
\end{defi}

\begin{exa} \label{exa: smf}
Every right-monoidal category $\M$ has a canonical right-monoidal functor into the strict monoidal category $\End\M$ of endofunctors of $\M$. Define $\eL:\M\to\End\M$ by $\eL(M)N =M\smp N$. Then the natural transformation
\[
\eL (M)\eL(N)\longrarr{\gamma_{M,N,\under}}\eL(M\smp N)
\]
together with the arrow $\id_\M\rarr{\eta}\eL(R)$ is a right-monoidal structure on $\eL$. Unlike for monoidal categories when this functor is a strong monoidal embedding, for general $\M$ the functor $\eL$ is not even strong right-monoidal.

Similarly, the functor $\eR(M)N=N\smp M$ has a right-opmonoidal structure as a functor $\M\to\End^\op\M$, to the category $\End\M$ equipped with opposite composition as (strict) monoidal structure.
\end{exa}

Obviously, if both $\M$ and $\N$ are monoidal then the notions of left- and right-(op)monoidal functors coincide and they are precisely the usual (op)monoidal functors.

\begin{defi} \label{def: twist}
If $\smp$ and $\smp'$ are two right-monoidal structures on the same category $\M$ with the same unit object $R$ then a
twist from the $\smp$ structure to the $\smp'$-structure is a natural isomorphism $w_{M,N}:M\smp N\iso M\smp' N$
such that $\bra \id_\M,w,1_R\ket$ is a right-monoidal functor from $\M$ with $\smp'$ to $\M$ with $\smp$ structure. 
\end{defi}

One can define skew-(op)monoidal natural transformations although there is nothing `skew' in them, so we drop the adjective:
\begin{defi}
Let $F,G:\M\to\N$ be skew-monoidal functors. A monoidal natural transformation $\nu:F\to G$ is a natural transformation
of the underlying functors which satisfies
\begin{align}
\nu_{X\smp Y}\ci F_{X,Y}&=G_{X,Y}\ci(\nu_X\smp\nu_Y)\\
\nu_R\ci F_0&=G_0\,.
\end{align}
Opmonoidal transformations are similar transformations between skew-opmonoidal functors.
\end{defi}
The right-monoidal categories together with the right-(op)monoidal functors and (op)monoidal natural transformations
form the 2-category $\RightMonCat$ ($\RightOpmonCat$). Similar 2-categories can be defined for left-monoidal categories.

In ordinary monoidal categories tensoring with the unit object defines rather trivial monads and/or comonads. In the
skew-monoidal setting they are more interesting.
\begin{lem} \label{lem: T Q chi}
Let $\bra\M,\smp,R,\gamma,\eta,\eps\ket$ be a right-monoidal category and define $\mu_M:=(\eps_R\smp M)\ci\gamma_{R,R,M}$
and $\delta_M:=\gamma_{M,R,R}\ci(M\smp  \eta_R)$. Then 
\begin{align*}
T&=\bra R\smp\under,\mu,\eta\ket\\
Q&=\bra\under\smp R,\delta,\eps\ket
\end{align*}
are a monad and a comonad on $\M$, respectively, and $\chi_M:=\gamma_{R,M,R}$ is a (mixed) distributive law $\chi:TQ\to QT$.
\end{lem}
\begin{proof}
Inserting $M=N=R$ in (\ref{SMC2}), composing with $\eta_R$ and using naturality of $\eta$ we obtain
\begin{equation}\label{SMC6}
\gamma_{R,R,R}\ci(R\smp\eta_R)\ci\eta_R=(\eta_R\smp R)\ci\eta_R
\end{equation}
In a similar fashion we obtain
\begin{equation}\label{SMC7}
\eps_R\ci(\eps_R\smp R)\ci\gamma_{R,R,R}=\eps_R\ci(R\smp\eps_R)
\end{equation}
using (\ref{SMC3}).
Now we can verify associativity of $\mu$,
\begin{align*}
\mu_M\ci(R\smp\mu_M)&=(\eps_R\smp M)\ci\gamma_{R,R,M}\ci(R\smp(\eps_R\smp M))\ci(R\smp\gamma_{R,R,M})=\\
&=(\eps_R\smp M)\ci((R\smp\eps_R)\smp M)\ci\gamma_{R,R\smp R,M}\ci(R\smp\gamma_{R,R,M})=\\
&\eqby{SMC7}(\eps_R\smp M)\ci((\eps_R\smp R)\smp M)\ci(\gamma_{R,R,R}\smp M)\ci
\gamma_{R,R\smp R,M}\ci(R\smp\gamma_{R,R,M})=\\
&\eqby{SMC1}(\eps_R\smp M)\ci((\eps_R\smp R)\smp M)\ci\gamma_{R\smp R,R,M}\ci\gamma_{R,R,R\smp M}=\\
&=(\eps_R\smp M)\ci \gamma_{R,R,M}\ci(\eps_R\smp(R\smp M))\ci\gamma_{R,R,R\smp M}=\\
&=\mu_M\ci\mu_{R\smp M}
\end{align*}
and coassociativity of $\delta$,
\begin{align*}
(\delta_M\smp R)\ci\delta_M&=(\gamma_{M,R,R}\smp R)\ci((M\smp\eta_R)\smp R)\ci\gamma_{M,R,R}\ci(M\smp \eta_R)=\\
&=(\gamma_{M,R,R}\smp R)\ci\gamma_{M,R\smp R,R}\ci(M\smp(\eta_R\smp R))\ci(M\smp \eta_R)=\\
&\eqby{SMC6}(\gamma_{M,R,R}\smp R)\ci\gamma_{M,R\smp R,R}\ci(M\smp\gamma_{R,R,R})\ci(M\smp(R\smp\eta_R))\ci
(M\smp\eta_R)=\\
&\eqby{SMC1}\gamma_{M\smp R,R,R}\ci\gamma_{M,R,R\smp R}\ci(M\smp(R\smp\eta_R))\ci(M\smp\eta_R)=\\
&=\gamma_{M\smp R,R,R}\ci((M\smp R)\smp\eta_R)\ci\gamma_{M,R,R}\ci(M\smp\eta_R)=\\
&=\delta_{M\smp R}\ci\delta_M\,.
\end{align*}
As for the left and right unit and counit equations
\begin{align}
\label{SMC10}
\mu_N\ci\eta_{R\smp N}&=R\smp N\\
\label{SMC11}
\mu_N\ci(R\smp\eta_N)&=R\smp N\\
\label{SMC12}
\eps_{M\smp R}\ci\delta_M&=M\smp R\\
\label{SMC13}
(\eps_M\smp R)\ci\delta_M&=M\smp R
\end{align}
notice that inserting $M=R$ in (\ref{SMC4}) we obtain (\ref{SMC11}), inserting $N=R$ in (\ref{SMC4}) we obtain (\ref{SMC13}),
inserting $M=R$ in (\ref{SMC2}) and composing with $\eps_R\smp N$ we obtain (\ref{SMC10}) and inserting $N=R$ in
(\ref{SMC3}) and composing with $M\smp\eta_R$ we obtain (\ref{SMC12}).

It remains to show that $\chi$ is a distributive law in the sense of the equations
\begin{align}
\label{SMC14}
(\mu_M\smp R)\ci\chi_{R\smp M}\ci(R\smp\chi_M)&=\chi_M\ci\mu_{M\smp R}\\
\label{SMC15}
(\chi_M\smp R)\ci\chi_{M\smp R}\ci(R\smp\delta_M)&=\delta_{R\smp M}\ci\chi_M\\
\label{SMC16}
\chi_M\ci\eta_{M\smp R}&=\eta_M\smp R\\
\label{SMC17}
\eps_{R\smp M}\ci\chi_M&=R\smp\eps_M\,.
\end{align}
Equations (\ref{SMC14}) and (\ref{SMC15}) are simple consequences of the pentagon (\ref{SMC1}) while (\ref{SMC16}) and (\ref{SMC17}) follow trivially from (\ref{SMC2}) and (\ref{SMC3}), respectively.
\end{proof}

The monad $T$ and the comonad $Q$ on the right-monoidal category $\M$ will be called the \textit{canonical monad} and the \textit{canonical comonad} of $\M$. For left monoidal categories they are $T=\under \smp R$ and $Q=R\smp\under$.

\begin{lem} \label{lem: smf}
If $\bra F,F_2,F_0\ket$ is a right-monoidal functor $\M\to\N$ then the pair $\bra F,\varphi\ket$, where $\varphi_M:=
F_{R,M}\ci (F_0\smp FM)$, is a monad morphism from the canonical monad $T$ of $\M$ to the canonical monad $T$ on $\N$, i.e.,
\begin{align}
F\mu\ci\varphi T\ci T\varphi&=\varphi\ci\mu F\\
F\eta&=\varphi\ci\eta F\,.
\end{align}
Dually, if $\bra F,F^2,F^0\ket$ is a right-opmonoidal functor $\M\to\N$ then the pair $\bra F,\psi\ket$, where $\psi_M:=
(FM\smp F^0)\ci F^{M,R}$, is a comonad morphism from the canonical comonad $Q$ of $\M$ to the canonical comonad $Q$ of $\N$.
\end{lem}
\begin{proof}
The statement for the monad morphism can be easily shown using the definition of $\mu$ and the right-monoidal functor axioms 
(\ref{smf1}), (\ref{smf2}) and (\ref{smf3}). The statement for the comonad morphism is then obtained by passing to the
dual right-monoidal category $\M^{\op,\rev}$.
\end{proof}

\begin{rmk}
If we want to formulate a bialgebra-like compatibility condition between $\mu$ and $\delta$ then here is a commutative 
diagram
\begin{equation} \label{compatibility 1}
\begin{CD}
R\smp(R\smp R)@>\mu_R>>R\smp R@>\delta_R>>(R\smp R)\smp R\\
@V{R\smp\delta_R}VV @. @AA{\mu_R\smp R}A\\
R\smp((R\smp R)\smp R)@.@.(R\smp (R\smp R))\smp R\\
@V{\delta_{R,Q^2 R}}VV@.@AA{\mu_{T^2 R,R}}A\\
(R\smp R)\smp((R\smp R)\smp R)@.\longrarr{\sigma_{R,R\smp R,R}}@.(R\smp (R\smp R))\smp(R\smp R)
\end{CD}
\end{equation}
where 
\[
\sigma_{L,M,N}:=((L\smp M)\smp\eta_N)\ci\gamma_{L,M,N}\ci(\eps_L\smp(M\smp N))\ :\ (L\smp R)\smp(M\smp N)\to
(L\smp M)\smp(R\smp N)
\]
and where the 2-argument $\delta$ and $\mu$ are defined by
\begin{align}
\label{delta_K,L}
\delta_{K,L}&:=\gamma_{K,R,L}\ci(K\smp \eta_L)\ :\ K\smp L\to QK\smp L\\
\label{mu_K,L}
\mu_{K,L}&:=(\eps_K\smp L)\ci\gamma_{K,R,L}\ :K\smp TL\to K\smp L\,.
\end{align}
They obey the relations
\begin{alignat}{2}
\delta_{QK,L}\ci\delta_{K,L}&=(\delta_K\smp L)\ci\delta_{K,L}&\qquad (\eps_K\smp L)\ci\delta_{K,L}&=K\smp L\\
\mu_{K,L}\ci\mu_{K,TL}&=\mu_{K,L}\ci(K\smp\mu_L)&\qquad\mu_{K,L}\ci(K\smp\eta_L)&=K\smp L\,.
\end{alignat}
Although diagram (\ref{compatibility 1}) is reminiscent to the compatibility condition between multiplication and comultiplication
of a bialgebroid, 
in order to confirm this interpretation one should investigate in which sense $\sigma$ is a generalized braiding, if at all.
\end{rmk}

\begin{rmk}
The composite $\delta_R\ci\mu_R$ is built from $\gamma$, $\eta$, $\delta$ and identity arrows and has the same source and
target as $\gamma_{R,R,R}$. But there is no sign that they would be equal. Instead, 
\[
\delta_R\ci\mu_R=(\mu_R\smp R)\ci\gamma_{R,R\smp R,R}\ci(R\smp\delta_R),
\]
that is to say $\chi_{R\smp R}$ fits into diagram (\ref{compatibility 1}) as a second row.
So coherence for skew-monoidal categories is expected to fail in its naive form.
\end{rmk}

\begin{rmk}
Using the notations (\ref{delta_K,L}), (\ref{mu_K,L}) there is an identity in any right-monoidal category:
\[
\mu_{R\smp R,R}\ci\delta_{R,R\smp R}\ =\ \gamma_{R,R,R}\,.
\]
More generally, we have
\[
\mu_{QM,N}\ci\delta_{M,TN}\ =\ \gamma_{M,R,N}\,,\qquad M,N\in\M.
\]
This result suggests that we should think of the skew-associator $\gamma$ as the Galois map of the `underlying' quantum groupoid of $\M$ even if there is no such a quantum groupoid in general.
\end{rmk}

\section{The motivating example: bialgebroids} \label{s: bgd}

Let $\Ab_R$ denote the category of right $R$-modules over the ring $R$. This category has no (obvious) monoidal structure.
But every $R$-bialgebroid defines a right-monoidal structure on $\Ab_R$ as we shall see below.

Let $H$ be a right $R$-bialgebroid with $R^\op\ot R$-ring and $R$-coring structure
\begin{align}
t^\sH\ot s^\sH&:R^\op\ot R\ \to\ H\\
\Delta^\sH&:H\ \to\ H\am{R_1} H\,.
\end{align}
The unit element of $H$ is denoted by $1^\sH$ and the counit $H\to R$ by $\eps^\sH$.
Then $H$ carries two left and two right actions of $R$ defined by
\begin{alignat*}{2}
\lambda_1(r)(h)&:=ht^\sH(r)&\qquad\qquad \rho_1(r)(h)&:=t^\sH(r)h\\
\lambda_2(r)(h)&:=s^\sH(r)h&\qquad\qquad \rho_2(r)(h)&:=hs^\sH(r)
\end{alignat*}
for $r\in R$, $h\in H$. The codomain $H\am{R_1} H$ of the comultiplication $\Delta^\sH$ is the tensor square w.r.t. $\rho_2$ and $\lambda_1$. 

For right $R$-modules $M$ and $N$ we introduce 
\begin{equation} \label{smp-bgd}
M\smp N\ :=\ M\am{R_1} (N\am{R_2}H)
\end{equation}
where $L\am{R_i}\under$ refers to tensoring over $R$ with respect to the $\lambda_i$ left action on $H$.
The result $M\smp N$ is considered as a right $R$-module w.r.t. the $\rho_2$ right action on $H$.
Elements of $M\smp N$ are denoted by $[m,n,h]$ instead of $m\ot(n\ot h)$. They therefore obey the relations
\begin{align*}
[m\cdot r,n,h]&=[m,n,ht^\sH(r)]\\
[m,n\cdot r,h]&=[m,n,s^\sH(r)h]\\
[m,n,h]\cdot r&=[m,n,hs^\sH(r)]
\end{align*}
so the following natural transformations are well-defined:
\begin{alignat*}{2}
\eta_M\ :\ M&\to R\smp M,&\quad \eta_M(m)&=[1^\sR,m,1^\sH]\\
\eps_M\ :\ M\smp R&\to M,&\quad \eps([m,r,h])&=m\cdot\eps^\sH(s^\sH(r)h)\\
\gamma_{L,M,N}\ :\ L\smp(M\smp N)&\to(L\smp M)\smp N,&\quad \gamma_{L,M,N}([l,[m,n,g],h])&=[[l,m,h\oneT],n,gh\twoT].
\end{alignat*}
It is easy to verify, using the bialgebroid axioms, that $\bra\Ab_R,\smp, R_R, \gamma,\eta,\eps\ket$ is a right-monoidal category.

One can notice that the skew-associator $\gamma$, which is uniquely determined by $\gamma_{R,R,R}$, is, up to isomorphisms
$R\smp(R\smp R)\cong H\am{R_2}H$ and $(R\smp R)\smp R\cong H\am{R_1}H$, the canonical map or Galois map
\[
H\am{R_2}H\ \to\ H\am{R_1}H,\qquad g\ot h\ \mapsto\ h\oneT\ot g h\twoT
\]
of $H$ as a left $H$-comodule algebra.
Therefore the bialgebroid is a Hopf algebroid (or $\times_R$-Hopf algebra) in the sense of \cite{Sch} precisely when
the skew-associator $\gamma$ is invertible.

\section{$E$-objects} \label{s: E-obj}

Let $E=\End R$ be the endomorphism monoid of the right-monoidal unit $R$. An $E$-object in $\M$ is an object $M$ together
with a morphism $\lambda_M:E\to\M(M,M)$ of monoids.
The category $\E$ of $E$-objects in $\M$ has arrows $M\to N$ the arrows $t\in\M(M,N)$ which satisfy $t\ci\lambda_M(r)=\lambda_N(r)\ci t$ for all $r\in E$.

Since the category of $E$-objects in $\Ab_R$ is the category of bimodules, $_R\Ab_R$, hence monoidal, we would like to see
if this category inherits a skew-monoidal structure from the one given on $\Ab_R$. This is the first step on the path going from
skew-monoidal structures on $\Ab_R$ to bialgebroids.

One can define the category of $E^{\ot m}$-$E^{\ot n}$-bimodules in $\M$ as the category of objects equipped with $m$ left $E$-actions and $n$ right $E$-actions that pairwise commute with each other.  Such objects will be called $(m,n)$-type $E$-objects.

\begin{lem} \label{lem: * of E-objects}
If $K$ and $L$ are left $E$-objects (i.e., they are $(1,0)$-type) then $K\smp L$ is a $(2,1)$-type $E$-object with
\begin{align*}
\lambda_1(r)&=\lambda_K(r)\smp L\\
\lambda_2(r)&=K\smp\lambda_L(r)\\
\rho_1(r)&=(\eps_K\smp L)\ci\gamma_{K,R,L}\ci(K\smp(r\smp L))\ci(K\smp\eta_L).
\end{align*}
More generally, if $K$ is an $E$-object of $(m_1,n_1)$-type and $L$ is of $(m_2,n_2)$-type then $K\smp L$ is an $E$-object of $(m_1+m_2,n_1+1+n_2)$-type.
\end{lem}
\begin{proof}
$\lambda_1$ and $\lambda_2$ are obviously left actions and commute with each other. $\rho_1$ is natural in $K\in\M$ and $L\in\M$ therefore it commutes with both $\lambda_1$ and $\lambda_2$ and also with any other left or right actions the objects $K$ or $L$ may possess. Therefore the statement follows immediately if we prove that the formula for $\rho_1$ defines a right action. Unitalness $\rho(R)=K\smp L$ follows directly from (\ref{SMC4}). As for multiplicativity 
\begin{align*}
\rho_1(r_1)\ci\rho_1(r_2)&=(\eps_K\smp L)\ci\gamma_{K,R,L}\ci(K\smp(r_1\smp L))\ci(\eps_K\smp\eta_L)\ci \gamma_{K,R,L}\ci(K\smp(r_2\smp L))\ci(K\smp\eta_L)=\\
&=(\eps_K\smp L)\ci((\eps_K\smp R)\smp L)\ci \gamma_{K\smp R,R,L}\ci \gamma_{K,R,R\smp L}\ci(K\smp(r_2\smp (r_1\smp L)))\ci\\
&\qquad (K\smp (R\smp \eta_L))\ci(K\smp\eta_L)=\\
&\eqby{SMC1}(\eps_K\smp L)\ci(\eps_{K\smp R}\smp L)\ci(\gamma_{K,R,R}\smp L)\ci\gamma_{K,R\smp R,L}\ci
(K\smp\gamma_{R,R,L})\ci\\
&\qquad (K\smp(r_2\smp (r_1\smp L)))\ci(K\smp\eta_{R\smp L})\ci(K\smp\eta_L)=\\
&=(\eps_K\smp L)\ci(\eps_{K\smp R}\smp L)\ci(\gamma_{K,R,R}\smp L)\ci((K\smp(r_2\smp R))\smp L)\ci \gamma_{K,R\smp R,L}\ci\\
&\qquad (K\smp((R\smp r_1)\smp L)\ci(K\smp\gamma_{R,R,L})\ci(K\smp\eta_{R\smp L})\ci(K\smp\eta_L)=\\
&\eqby{SMC3}(\eps_K\smp L)\ci((K\smp\eps_R)\smp L)\ci((K\smp (r_2\smp R))\smp L)\ci \gamma_{K,R\smp R,L}\ci\\
&\qquad (K\smp((R\smp r_1)\smp L)\ci(K\smp\gamma_{R,R,L})\ci(K\smp\eta_{R\smp L})\ci(K\smp\eta_L)=\\
&\eqby{SMC2}(\eps_K\smp L)\ci((K\smp\eps_R)\smp L)\ci((K\smp (r_2\smp R))\smp L)\ci \gamma_{K,R\smp R,L}\ci
\end{align*}
\begin{align*}
&\qquad (K\smp((R\smp r_1)\smp L)\ci(K\smp(\eta_R\smp L))\ci(K\smp\eta_L)=\\
&=(\eps_K\smp L)\ci((K\smp r_2)\smp L)\ci ((K\smp\eps_R)\smp L) \ci \gamma_{K,R\smp R,L}\ci (K\smp(\eta_R\smp L))\ci\\
&\qquad (K\smp(r_1\smp L))\ci(K\smp\eta_L)=\\
&\eqby{SMC5}(\eps_K\smp L)\ci\gamma_{K,R,L}\ci(K\smp((r_2\ci r_1)\smp L))\ci(K\smp\eta_L)=\\
&=\rho_1(r_2\ci r_1).
\end{align*}
This completes the proof.
\end{proof}
If we have $n$ left $E$-objects and we $\smp$ them in any order, so the parenthesizing is arbitrary, then the resulting object
will have $n$ left actions of the obvious $1\smp\dots 1\smp\lambda(r)\smp 1\smp\dots\smp 1$ type and less obvious right actions,
$n-1$ in number, each corresponding to one $\smp$ sign. These actions will be numbered from left to right as shown:
\[
\underset{\lambda_1}{A}\ \overset{\rho_1}{\smp}\ \underset{\lambda_2}{B}\ \overset{\rho_2}{\smp}\ \dots\ 
\overset{\rho_{n-1}}{\smp}\ \underset{\lambda_n}{Z}
\]

The simplest left $E$-object is $R$. Its left action is the identity morphism $E\to\M(R,R)$. By the above Lemma the object $R\smp R$
is equipped with two left actions $\lambda_1$, $\lambda_2$ and one right action $\rho_1$. As such a $(2,1)$-type object 
$R\smp R$ is denoted by $H$. It is to be interpreted as the underlying object of a quantum groupoid, at least for $\M=\Ab_R$. 

In the next Lemma we summarize how the structure maps $\gamma$, $\eta$, $\eps$ and their derivatives $\mu$ and $\delta$  behave with respect to the $\lambda$ and $\rho$ actions.
\begin{lem} \label{lem: lambda-rho}
For $E$-objects $L$, $M$, $N$ and for all $r\in E$
\begin{align}
\label{lambda-rho-1}
\lambda_i(r)\ci\gamma_{L,M,N}&=\gamma_{L,M,N}\ci\lambda_i(r)\qquad i=1,2,3\\
\label{lambda-rho-2}
\lambda_2(r)\ci\eta_N&=\eta_N\ci\lambda_1(r)\\
\lambda_1(r)\ci\eps_L&=\eps_L\ci\lambda_1(r)\\
\lambda_1(r)\ci\mu_N&=\mu_N\ci\lambda_1(r)\\
\label{lambda-rho-5}
\lambda_2(r)\ci\mu_N&=\mu_N\ci\lambda_3(r)\\
\lambda_3(r)\ci\delta_L&=\delta_L\ci\lambda_2(r)\\
\lambda_1(r)\ci\delta_L&=\delta_L\ci\lambda_1(r)\,.
\end{align}
For arbitrary $L$, $M$, $N$ of $\M$ and for all $r\in E$
\begin{align}
\label{lambda-rho-8}
\rho_i(r)\ci\gamma_{L,M,N}&=\gamma_{L,M,N}\ci\rho_i(r)\qquad i=1,2\\
\label{lambda-rho-9}
\rho_1(r)\ci\eta_N&=\lambda_1(r)\ci\eta_N\\
\eps_L\ci \rho_1(r)&=\eps_L\ci\lambda_2(r)\\
\label{lambda-rho-11}
\rho_1(r)\ci\mu_N&=\mu_N\ci\rho_2(r)\\
\rho_1(r)\ci\delta_L&=\delta_L\ci\rho_1(r)\\        
\label{lambda-rho-13}
\mu_N\ci\rho_1(r)&=\mu_N\ci\lambda_2(r)\\
\rho_2(r)\ci\delta_L&=\lambda_2(r)\ci\delta_L\,.
\end{align}
\end{lem}
\begin{proof}
Relations involving $\lambda$-s only are just naturalities of the structure maps. Those involving $\rho$-s require some computations which, however, are left to the reader.
\end{proof}

Among the various multiple $E$-objects there are distinguished ones that behave nicely under the $\smp$-product. For each $n>0$
let $\M^{(n)}$ denote the category of $(n,n-1)$-type of $E$-objects in $\M$. Then $\M^{(m)}\smp\M^{(n)}\subset\M^{(m+n)}$ by
Lemma \ref{lem: * of E-objects}. Clearly, $\M^{(1)}=\E$ and $R\in\M^{(1)}$, $H\in\M^{(2)}$. The coproduct 
$\M^{(\bullet)}=\bigsqcup_{n>0}\M^{(n)}$ is then closed under $\smp$ but has no unit object. 

Now assume that the category $\M$ has limits and colimits. For two left $E$-objects $L$ and $M$ we can make new $E$-objects from the $(2,1)$-type object $L\smp M$ either by forming the $\lambda_1$-$\rho_1$ center or by forming the $\rho_1$-$\lambda_2$ quotient:
\begin{align}
\label{center}
\int_{\lambda_1 \rho_1}L\smp M&\equalizer{z_{L,M}}L\smp M\pair{\lambda_1}{\rho_1}\prod_{r\in E}L\smp M\\
\label{quotient}
\coprod_{r\in E} L\smp M\pair{\rho_1}{\lambda_2}L\smp M&\coequalizer{q_{L,M}}\int^{\rho_1 \lambda_2}L\smp M
\end{align}
Then the $\lambda_2$ action on $L\smp M$ inherits to $\int_{\lambda_1\rho_1}L\smp M$ a left $E$-object structure and $\lambda_1$ inherits one to $\int^{\rho_1\lambda_2}L\smp M$.  In this way, the above end and coend define functors $\E\times\,\E\to\,\E$. The identity arrow on $L\smp M$ restricts-corestricts to a natural transformation 
\begin{equation}
\label{theta}
\theta_{L,M}:=q_{L.M}\ci z_{L,M}\ :\quad\int_{\lambda_1 \rho_1}L\smp M\ \to\ \int^{\rho_1\lambda_2}L\smp M\ .
\end{equation}
Indeed, for $r\in E$
\[
\theta_{L,M}\ci\lambda_2(r)=q_{L,M}\ci\lambda_2(r)\ci z_{L,M}=q_{L,M}\ci\rho_1(r)\ci z_{L,M}=
q_{L,M}\ci\lambda_1(r)\ci z_{L,M}=\lambda_1(r)\ci\theta_{L,M}
\]
shows that $\theta_{L,M}$ belongs to $\E$. Its naturality follows from that $z$ and $q$ are natural.

\begin{pro} \label{pro: gamma+}
Let $\bra\M,\smp,R,\gamma,\eta,\eps\ket$ be a right-monoidal category in which the category $\M$ has colimits and $L\smp\under:\M\to\M$ preserves finite colimits for each $L\in\M$. Choosing a coequalizer (\ref{quotient}) for each pair of
$E$-objects $\bra L,M\ket$ and making the quotient 
\[
L\smpq M :=\int^{\rho_1\lambda_2} L\smp M
\]
an $E$-object by means of $\lambda_1$ there is a unique right-monoidal structure
$\bra\,\E,\smpq,R,\gamma^q,\eta^q,\eps^q\ket$ on the category of $E$-objects such that the forgetful functor $\phi:\,\E\to\M$ together with $q_{L,M}:L\smp M\to L\smpq M$ and the identity arrow $1_R$ becomes a right-monoidal functor
$\E\to\M$.
\end{pro}
\begin{proof}
For $\bra\phi,q,1_R\ket$ to be a right-monoidal functor the $\gamma^q$, $\eta^q$ and $\eps^q$ must obey to commutativity of
the diagrams
\begin{gather} \label{q-hexa}
\begin{CD}
L\smp(M\smp N)@>L\smp q_{M,N}>>L\smp(M\smpq N)@>q_{L,M\smpq N}>>L\smpq(M\smpq N)\\
@V{\gamma_{L,M,N}}VV @. @VV{\gamma^q_{L,M,N}}V\\
(L\smp M)\smp N@>q_{L,M}\smp N>>(L\smpq M)\smp N@>q_{L\smpq M,N}>>(L\smpq M)\smpq N
\end{CD}
\\  \label{q-squares}
\begin{CD}
M@>\eta_M>>R\smp M\\
@| @VV{q_{R,M}}V\\
M@>\eta^q_M>>R\smpq M
\end{CD}
\qquad\qquad
\begin{CD}
M\smp R@>\eps_M>> M\\
@V{q_{M,R}}VV @|\\
M\smpq R@>\eps^q_M>>M
\end{CD}
\end{gather}
The existence and uniqueness of $\gamma^q$ follow from that the composite $\xi:=
q_{L\smpq M, N}\ci(q_{L,M}\smp N)\ci\gamma_{L,M,N}$ satisfies both $\xi\ci\rho_1=\xi\ci\lambda_2$ and
$\xi\ci\rho_2=\xi\ci\lambda_3$ as a consequence of (\ref{lambda-rho-1}), (\ref{lambda-rho-8}). 
By the latter there is a unique factorization 
$\xi=\xi'\ci(L\smp q_{M,N})$ in which $\xi'\ci\rho_1=\xi'\ci\lambda_2$. Then $\gamma^q$ is obtained as the unique factorization
$\xi'=\gamma^q_{L,M,N}\ci q_{L,M\smpq N}$.
$\eps^q$ is obtained in a similar way while $\eta^q$ is readily defined by the diagram as it stands.
  
The verification of the right-monoidal category axioms is now a routine computation.
\end{proof}

The dual of Proposition \ref{pro: gamma+} is the following
\begin{pro} \label{pro: gammax}
Let $\bra\M,\smp,R,\gamma,\eta,\eps\ket$ be a right-monoidal category in which the category $\M$ has limits and 
$\under\smp M:\M\to\M$ preserves finite limits for each $M\in\M$. Choosing an equalizer (\ref{center}) for each pair of
$E$-objects $\bra L,M\ket$ and making the center 
\[
L\smpz M :=\int_{\lambda_1\rho_1} L\smp M
\]
an $E$-object by means of $\lambda_2$ there is a unique right-monoidal structure
$\bra\,\E,\smpz,R,\gamma^z,\eta^z,\eps^z\ket$ on the category of $E$-objects such that
$\bra\phi,z,1_R\ket$ is a right-opmonoidal functor, i.e.,
\begin{gather}
\begin{CD}
L\smpz(M\smpz N)@>z_{L,M\smpz N}>>L\smp(M\smpz N)@>L\smp z_{M,N}>>L\smp(M\smp N)\\
@V{\gamma^z_{L,M,N}}VV @. @VV{\gamma_{L,M,N}}V\\
(L\smpz M)\smpz N@>z_{L\smpz M,N}>>(L\smpz M)\smp N@>z_{L,M}\smp N>>(L\smp M)\smp N
\end{CD}
\\ 
\begin{CD}
M@>\eta^z_M>>R\smpz M\\
@| @VV{z_{R,M}}V\\
M@>\eta_M>>R\smp M
\end{CD}
\qquad\qquad
\begin{CD}
M\smpz R@>\eps^z_M>> M\\
@V{z_{M,R}}VV @|\\
M\smp R@>\eps_M>>M
\end{CD}
\end{gather}
are commutative for each $L,M,N\in\,\E$.
\end{pro}

Applying Lemma \ref{lem: smf} to the skew-(op)monoidal functor of Proposition \ref{pro: gamma+} and Proposition 
\ref{pro: gammax}, respectively, we obtain the following.
\begin{cor} \label{cor: (co)monad morphisms}
Let $\bra\M,\smp,R,\gamma,\eta,\eps\ket$ be a right-monoidal category with canonical monad $T$ and canonical comonad $Q$.
\begin{enumerate}
\item If $\M$ has colimits and for all $L\in\M$ the endofunctor $L\smp\under$ preserves finite colimits then
\begin{enumerate}
\item $\E$ has a right-monoidal structure with canonical monad $T_q=\int^{\rho_1\lambda_2}R\smp\under$
\item and $\kappa_M:=q_{R,M}$ defines a monad morphism $\bra\phi,\kappa\ket$ from $T_q$ to $T$.
\end{enumerate}
\item If $\M$ has limits and for all $M\in\M$ the endofunctor $\under\smp M$ preserves finite limits then
\begin{enumerate}
\item $\E$ has a right-monoidal structure with canonical comonad $Q^z=\int_{\lambda_1\rho_1}\under\smp R$
\item and $\zeta_L:=z_{L,R}$ defines a comonad morphism $\bra\phi,\zeta\ket$ from $Q^z$ to $Q$.
\end{enumerate}
\end{enumerate}
\end{cor}

As we shall see in the next section some results of this Corollary hold under weaker hypotheses.

\section{Comodules and modules}

If right-monoidal categories are to be interpreted as quantum groupoids then it must have associated categories of modules and
comodules. The Eilenberg-Moore categories of the canonical monad $T$ and comonad $Q$ are the obvious candidates, albeit
apparently without monoidal structures.

Let $\M^Q$ denote the Eilenberg-Moore category of $Q$-comodules, also called $Q$-coalgebras, for the comonad $Q=\bra\under\smp R,\delta,\eps\ket$. Its objects are pairs $\bra M,\Delta_M\ket$ where $M$ is an object of $M$ and 
$\Delta_M:M\to M\smp R$ satisfies 
\begin{align}
\label{coact1}
(\Delta_M\smp R)\ci\Delta_M&=\delta_M\ci\Delta_M\\
\label{coact2}
\eps_M\ci\Delta_M&=M\,.
\end{align}
The arrows $M\to N$ in $\M^Q$ are defined to be the arrows $t\in\M(M,N)$ such that 
\begin{equation} \label{comodmap}
\Delta_N\ci t=(t\smp R)\ci\Delta_M\,.
\end{equation}

Dually, in the category $\M_T$ of $T$-modules the objects $\nabla_M:R\smp M\to M$ are defined by the equations
\begin{align}
\label{act1}
\nabla_M\ci(R\smp\nabla_M)&=\nabla_M\ci\mu_M\\
\label{act2}
\nabla_M\ci\eta_M&=M\,.
\end{align}
and its arrows $t:M\to N$ by
\begin{equation} \label{modmap}
t\ci\nabla_M=\nabla_N\ci(R\smp t)\,.
\end{equation}

Entwined modules of a skew-monoidal category can be defined as the category of triples $\bra M,\nabla,\Delta\ket$ such that 
$\bra M,\nabla\ket$ is a $T$-module and $\bra M,\Delta\ket$ is a $Q$-comodule which satisfy the compatibility condition
\[
\begin{CD}
TM@>\nabla>>M@>\Delta>>QM\\
@V{T\Delta}VV @. @AA{Q\nabla}A\\
TQM@.\rarr{\chi_M}@.QTM
\end{CD}
\]
The arrows $\bra M,\Delta,\nabla\ket\to\bra M',\Delta',\nabla'\ket$ are the arrows $t\in\M(M,M')$ which are both $T$-module 
and $Q$-comodule morphisms. 
The basic example of an entwined module is the object $R\smp R$ with action $\mu_R$ and coaction $\delta_R$.

\begin{lem} \label{lem: forget to E-mod}
If $L$ is a $Q$-comodule and $N$ is a $T$-module 
then both $L$ and $N$ are left $E$-objects via
\begin{align}
\lambda_L(r)&=\eps_L\ci (L\smp r)\ci\Delta_L\\
\label{T-induced lambda}
\lambda_N(r)&=\nabla_N\ci(r\smp N)\ci\eta_N\,,
\end{align}
respectively. With respect to these actions every arrow in $\M^Q$ and every arrow in $\M_T$ are morphisms of left $E$-objects.
This defines the faithful functors
\[
\F_z:\M^Q\to\ \E\,,\qquad\F_q:\M_T\to\ \E
\]
\end{lem}
\begin{proof}
Since $T$-modules in $\M$ are the $Q$-comodules of the opposite-reversed right-monoidal category $\M^{\op,\rev}$, it suffices to show that $\lambda_L$ is a monoid morphism and that every $t\in\M^Q$ is a morphism of $E$-objects.
\begin{align*}
\lambda_L(R)&=\eps_L\ci\Delta_L\eqby{coact2}L\\
\lambda_L(r_1)\ci\lambda_L(r_2)&=\eps_L\ci\eps_{L\smp R}\ci((L\smp r_1)\smp R)\ci((L\smp R)\smp r_2)\ci 
(\Delta_L\smp R)\ci\Delta_L=\\
&\eqby{coact1}\eps_L\ci\eps_{L\smp R}\ci((L\smp r_1)\smp r_2)\ci \delta_L\ci\Delta_L=\\
&=\eps_L\ci\eps_{L\smp R}\ci\gamma_{L,R,R}\ci(L\smp (r_1\smp r_2))\ci (L\smp\eta_R)\ci\Delta_L=\\
&\eqby{SMC3}\eps_L\ci(L\smp \eps_R)\ci(L\smp (r_1\smp r_2))\ci (L\smp\eta_R)\ci\Delta_L=\\
&=\eps_L\ci(L\smp r_1)\ci(L\smp\eps_R)\ci(L\smp\eta_R)\ci(L\smp r_2)\ci\Delta_L=\\
&\eqby{SMC5}\eps_L\ci(L\smp (r_1\ci r_2))\ci\Delta_L=\\
&=\lambda_L(r_1\ci r_2)\,.
\end{align*}
If $t:K\to L$ is a $Q$-comodule morphism then
\[
t\ci\lambda_K(r)=\eps_L\ci(t\smp r)\ci\Delta_K=\eps_L\ci(L\smp r)\ci\Delta_L\ci t=\lambda_L\ci t\,.
\]
\end{proof}

We note that for the free $Q$-comodules $N\smp R\rarr{\delta_{N}}(N\smp R)\smp R$, where $N$ is an arbitrary object in $\M$,
the above left $E$-action $\lambda_{N\smp R}$ reduces to the canonical $N\smp r$ left action $\lambda_2$ of the right-monoidal product $N\smp R$ of a $(0,0)$-type object with a $(1,0)$-type object. Dually, for free $T$-modules 
$\lambda_{R\smp N}(r)=r\smp N$. However, if $L$ is a $Q$-comodule and $M$ is a $T$-module then $L\smp R$ and $R\smp M$ are type $(2,1)$ and the question arises how the coaction and action behave with respect to the extra two $E$-actions. 
\begin{lem} \label{lem: central coactions quotient actions}
Assume $\M$ is complete. For every $Q$-comodule $L$ the coaction $\Delta_L$ is a morphism of left $E$-objects
and factorizes uniquely through the center of the $(2,1)$-type $E$-object $L\smp R$ as
\[
L\longrarr{\Delta_L^z}\int_{\lambda_1 \rho_1}L\smp R\equalizer{z_{L,R}}L\smp R
\]
in $\E$. Dually, assume $\M$ is cocomplete. Then the action $\nabla_M$ of every $T$-module $M$ belongs to $\E$ and has a unique factorization
\[
R\smp M\coequalizer{q_{R,M}}\int^{\rho_1 \lambda_2} R\smp M\longrarr{\nabla_M^q}M
\]
in $\E$ through the quotient of the $(2,1)$-type $E$-object $R\smp M$.
\end{lem}
\begin{proof}
We prove the statement for $Q$-coactions. Since every comodule $L$ is an equalizer
\[
L\equalizer{\Delta_L}L\smp R\pair{\delta_L}{\Delta_L\smp R}(L\smp R)\smp R
\]
in $\M$ (it is split by $L\larr{\eps_L}L\smp R\larr{\eps_{L\smp R}}(L\smp R)\smp R$), the coaction $\Delta_L$ is a morphism of $Q$-comodules from $L$ to the free $Q$-comodule $L\smp R$. Therefore by Lemma \ref{lem: forget to E-mod} it is also a morphism
of $E$-objects with respect to the $\lambda_2$ action on $L\smp R$, i.e.,
\[
\Delta_L\ci\lambda_L(r)=(L\smp r)\ci\Delta_L\,,\qquad r\in E\,.
\]
As for the remaining two actions we can compute, using the expressions in Lemma \ref{lem: * of E-objects} for $\rho_1$,
$\lambda_1$, that
\begin{align*}
\rho_1(r)\ci\Delta_L&=(\eps_L\smp R)\ci((L\smp r)\smp R)\ci\delta_L\ci\Delta_L=\\
&=(\eps_L\smp R)\ci((L\smp r)\smp R)\ci(\Delta_L\ci R)\ci\Delta_L=(\lambda_L(r)\smp R)\ci\Delta_L=\\
&=\lambda_1(r)\ci\Delta_L
\end{align*}
from which the unique factorization through $z_{L,R}\in\,\E$ follows. 
\end{proof}

Note that in the above Lemma we avoided to use the notation $\smpq$ and $\smpz$ because under the given conditions they need not be skew-monoidal products.

\begin{thm} \label{thm: equivalence of module cats}
If $\M$ has colimits and the endofunctor $R\smp\under$ preserves coequalizers then
\begin{enumerate}
\item the endofunctor $M\mapsto T_q M:=\int^{\rho_1\lambda_2} R\smp M$ on $\E$ carries a unique monad structure
such that the forgetful functor $\phi:\,\E\to\M$ together with the coequalizer 
$T\phi M\overset{\kappa_M}{\twoheadrightarrow}\phi T_qM$ of $\rho_1$ 
and $\lambda_2$ is a monad morphism $\bra\phi,\kappa\ket$ from $T_q$ to $T$;
\item the functor $\phi_q$ induced by the monad morphism $\bra\phi,\kappa\ket$ is an equivalence of the Eilenberg-Moore categories such that
\begin{equation} \label{phi_q square}
\begin{CD}
\E_{T_q}@>\phi_q>>\M_T\\
@V{\F_{T_q}}VV @VV{\F_T}V\\
\E@>\phi>>\M
\end{CD}
\end{equation}
and the functor $\F_q:\M_T\to\,\E$ of Lemma \ref{lem: forget to E-mod} is monadic and satisfies
\begin{equation} \label{F_q triangles}
\F_q\phi_q\ =\ \F_{T_q}\,,\qquad \phi\F_q\ =\ \F_T\,.
\end{equation}
\end{enumerate}
\end{thm}
\begin{proof} 
This Theorem follows by dualizing the next Theorem \ref{thm: equivalence of comodule cats} 
\end{proof}

\begin{thm} \label{thm: equivalence of comodule cats} 
If $\M$ has limits and the endofunctor $\under\smp R$ preserves equalizers then 
\begin{enumerate}
\item the endofunctor $M\mapsto Q^z M:=\int_{\lambda_1\rho_1} M\smp R$ on $\E$ carries a unique comonad structure
such that the forgetful functor $\phi:\,\E\to\M$ together with the equalizer 
$\phi Q^zM\overset{\zeta_M}{\rightarrowtail}Q\phi M$ of $\lambda_1$ and $\rho_1$ 
is a comonad morphism $\bra\phi,\zeta\ket$ from $Q^z$ to $Q$;
\item the functor $\phi_z$ induced by the comonad morphism $\bra\phi,\zeta\ket$ is an equivalence of the Eilenberg-Moore categories such that
$\F^Q\phi_z=\phi\F^{Q^z}$ and the functor $\F_z:\M^Q\to\,\E$ of Lemma \ref{lem: forget to E-mod} is comonadic and satisfies
\[
\F_z\phi_z\ =\ \F^{Q^z}\,,\qquad \phi\F_z\ =\ \F^Q\,.
\]
\end{enumerate}
\end{thm}
\begin{proof}
This Theorem is the special case of the lax version proven in the next section. Part (i) follows from Proposition \ref{pro: bfQ} 
and part (ii) from Theorem \ref{thm: equivalence of lax comodule cats} 
after noticing that left exactness of $Q$ implies the possibility to choose the equalizers $\zeta^n$ in such a way that
$\bfQ_n=(Q^z)^n$ for each $n\geq 0$.
\end{proof}
\begin{exa}
For a right $R$-bialgebroid $H$ as in Section \ref{s: bgd} the monad $T$ is $\under\oR H$ associated to the $R$-ring
$R\rarr{s^\sH}H$ and $T_q$ is $\under\am{R^e}H$ associated to the $R^e$-ring $R^\op\ot R\rarr{t^\sH\ot s^\sH}H$.
The monad morphism $\kappa_M$ is the canonical projection $M\oR H\epi{}M\am{R^e}H$ and the fact that it induces an
equivalence between the corresponding right $H$-module categories can be considered as a well-known fact in the bialgebroid
literature and it is a consequence of the fact that $T$ is right exact. However, the dual statement Theorem 
\ref{thm: equivalence of comodule cats} presents a warning that the category $(\Ab_R)^H$ of right comodules over the $R$-coring
$H$ may not be equivalent to the Eilenberg-Moore category of the comonad $Q^z$ on $_R\Ab_R$ unless $_RH$ is flat, i.e., $Q$ is
left exact. This equivalence is crucial in Tannaka duality where we want $Q^z$ a monoidal comonad on the bimodule category
$_R\Ab_R$. Without left exactness the $Q^z$ will not even be a comonad. What replaces $Q^z$ in the general case is a lax comonad
discussed in the next section. 
\end{exa}

\section{The lax comonad $\bfQ$} \label{s: lax Q}

In \cite[Proposition 4.2]{Day-Street} Day and Street have characterized (left) $R$-bialgebroids as comonoids in the lax monoidal category of monads on $R^e$ where the lax monoidal structure is given by $n$-fold Takeuchi products 
$M_1\x_R\dots\x_R M_n$.
Here we shall concentrate on the closely related but simpler structure of monoidal lax comonads on the category $\E$ of 
$E$-objects but ignore monoidality altogether, as we did so far for $T$ and $T_q$, and be content with proving equivalence of 
$\M^Q$ with the category $\E^\bfQ$ of comodules for the lax comonad $\bfQ$ with the hope in mind that if $\E$ is provided a
`good' monoidal structure then $\E^\bfQ$ will become monoidal, too. 

Let $\Delta$ be the category of finite ordinals and order preserving maps equipped with the strict monoidal structure of ordinal
addition $+$. By a lax comonad on a category $\E$ we mean a monoidal functor $G:\Delta^\op\to \End \E$ to the strict monoidal category of endofunctors on $\E$ with composition of functors as monoidal product. The monoidal structure of $G$ is given by an `arrow' $\iota:\id_\E\to G_0$ of $\End\E$ and a natural transformation $\nu_{m,n}:G_mG_n\to G_{m+n}$ satisfying 3 axioms,
as usual. If the functor $G$ happens to be strict monoidal then the object map of $G$ is $G_n=(G_1)^n$ and we recover an ordinary
comonad $\bra G_1, G_{2\to 1},G_{0\to 1}\ket$ on $\E$. 

The generalization of the Eilenberg-Moore category for the lax situation goes as follows.
A comodule over a lax comonad $\bra \E,G\ket$ consists of an object $M$ of $\E$ and arrows $\alpha_n:M\to G_n M$ for 
each $n\geq 0$ such that
\begin{align*}
G_f\ci\alpha_n&=\alpha_m\quad \forall f:m\to n \\
\alpha_{m+n}&=\nu_{m,n}M\ci G_m\alpha_n\ci\alpha_m\quad \forall m,n\geq 0\\
\alpha_0&=\iota_M\,.
\end{align*}
A comodule map $\bra M,\alpha\ket\rarr{t}\bra N,\beta\ket$ is an arrow $M\rarr{t}N$ in $\E$ such that
\[
\begin{CD}
M@>t>>N\\
@V{\alpha_n}VV @VV{\beta_n}V\\
G_n M@>G_n t>>G_n N
\end{CD}
\qquad\forall n\geq 0\,.
\]
The category of $G$-comodules and their comodule maps is denoted by $\E^G$. The forgetful functor $\E^G\to\E$, $\bra M,\alpha\ket\mapsto M$ is faithful, reflects isomorphisms but not left adjoint in general.

In order to justify the above definition of $\E^G$ it is worth looking at its 2-categorical interpretation.
For lax comonads $F$ on $\D$ and $G$ on $\E$ a morphism of lax comonads $\bra \D,F\ket\to\bra \E,G\ket$ can be defined to 
consist of a functor $U:\D\to\E$ and natural transformations 
\[
\xi_n:UF_n\to G_nU :\D\to\E\qquad\text{natural in }n\in\Delta^\op
\]
and obeying the following monoidality conditions
\[
\begin{CD}
UF_mF_n@>\xi_mF_n>>G_mUF_n@>G_m\xi_n>>G_mG_nU\\
@VVV @.@VVV\\
UF_{m+n}@.\longrarr{\xi_{m+n}}@.G_{m+n}U
\end{CD}
\qquad
\begin{CD}
U@=U\\
@V{}VV @VV{}V\\
UF_0@>\xi_0>>G_0U
\end{CD}
\]
A modification $\tau:\bra U,\xi\ket\to\bra V,\upsilon\ket:\bra \D,F\ket\to\bra \E,G\ket$ is a natural transformation $\tau:U\to V$
satisfying
\[
\begin{CD}
UF_m@>\tau F_m>>VF_m\\
@V{\xi_m}VV @VV{\upsilon_m}V\\
G_m U@>G_m\tau>>G_mV
\end{CD}
\qquad\forall m\geq 0\,.
\]
With the obvious horizontal and vertical compositions the lax comonads, their morphisms and modifications form a 2-category
$\LaxCmd$. 

\begin{lem} \label{lem: lax-cmd one}
Let $\one$ be the identity comonad on the terminal category $1$. Then for any lax comonad $\bra \E,G\ket$ the Eilenberg-Moore category $\E^G$ of $G$-comodules can be identified with the hom-category $\LaxCmd(\bra 1,\one\ket,\bra\E,G\ket)$.
\end{lem}
\begin{proof}
A morphism of lax comonads from $\one$ to $G$ is an object $M$ of $\E$ equipped with $\alpha_n:M\to G_n M$, $n\geq 0$,
satisfying precisely the defining relations of a $G$-comodule. A modification $\bra M,\alpha\ket\rarr{t}\bra N,\beta\ket$ in turn
is an arrow $M\rarr{t}N$ satisfying $\beta_n\ci t=G_n t\ci\alpha_n$, $n\geq 0$, i.e., a comodule map.
\end{proof}

By extending Lemma \ref{lem: lax-cmd one} notice that a morphism $\bra U,\xi\ket:\bra \D,F\ket\to\bra \E,G\ket$ of lax comonads
induces a functor 
\[
\LaxCmd(\bra 1,\one\ket,\bra U,\xi\ket)\ :\ \D^F\to\E^G
\]
between the Eilenberg-Moore categories the object map of which is
\[
\bra D,\alpha\ket\mapsto\bra UD,(UD\rarr{U\alpha_n}UF_n D\rarr{\xi_n D}G_nUD)_{n\geq 0}\ket\,.
\]

After this preparation we can introduce the \textit{canonical lax comonad} $\bfQ$ of a right-monoidal category 
$\bra \M,\smp,R,\gamma,\eta,\eps\ket$. For an $E$-object $M$ we define $\bfQ_n M$ by delaying the action of the ends in $(Q^z)^n M$, i.e., by the formula
\[
\bfQ_n M:=\int_{\lambda_1\rho_1}\dots\int_{\lambda_n\rho_n}(\dots(M\smp R)\smp\dots\smp R)\smp R
\]
where the number of $R$-s is $n$ and the left and right $E$-actions $\lambda_i$, $\rho_i$ are labeled according to what we said
in Section \ref{s: E-obj}. The result $\bfQ_n M$ becomes a left $E$-object via $\lambda_{n+1}$ which is the action on the last $R$ factor. 

$\bfQ_0$ is the identity functor and $\bfQ_1$ is the endofunctor $Q^z$ of Corollary \ref{cor: (co)monad morphisms} (ii).
But now, without the assumption that $\under\smp R$ preserves equalizers, $Q^z$ does not inherit a comonad structure from 
that of $Q$. Although $\eps^z:\bfQ_1\to \bfQ_0$ exists we cannot define comultiplication $\bfQ_1\to\bfQ_1^2$. Instead we can
define a natural transformation $\delta^1_0:\bfQ_1\to\bfQ_2$. 

\begin{pro} \label{pro: bfQ}
Let $\M$ be a right-monoidal category whose underlying category $\M$ is complete. Let $\E$ be the category of $E$-objects
in $\M$, $\phi$ the forgetful functor $\E\to\M$ and define the endofunctors $\bfQ_n$ on $\E$ for $n\geq 0$ by the equalizers
\[
\phi\bfQ_n M\equalizer{\zeta^n_M}Q^n\phi M\longerpair{\bra\lambda_1,\dots,\lambda_n\ket}{\bra\rho_1,\dots,\rho_n\ket}
\{E^{\ot n},Q^n\phi M\}
\]
where $\{\ ,\ \}$ denotes cotensor (=power) in $\M$. Then $n\mapsto \bfQ_n$ is the object map of a unique lax comonad $\bfQ$
on $\E$ such that $\phi$ together with $\{\zeta^n|n\geq 0\}$ is a morphism of lax comonads $\bfQ\to Q$.
\end{pro}
\begin{proof}
In order to extend $\bfQ$ to a functor $\Delta^\op\to\End\E$ it suffices to define it on the elementary monotone functions
$i+(2\to 1)+j$ and $i+(0\to 1)+j$. Naturality of $\zeta^n$ determines them to be the unique $\delta^n_i:\bfQ_n\to\bfQ_{n+1}$
and $\eps^n_i:\bfQ_n\to\bfQ_{n-1}$, respectively, such that
\begin{align}
\label{def of delta^n_i}
\zeta^{n+1}\ci\delta^n_i&=Q^i\delta Q^{n-i-1}\ci\zeta^n\qquad i=0,1,\dots,n-1,\quad n\geq 0\\
\label{def of eps^n_i}
\zeta^{n-1}\ci\eps^n_i&=Q^i\eps Q^{n-i-1}\ci\zeta^n\qquad i=0,1,\dots,n-1,\quad n > 0.
\end{align}
For their existence the reader should check that the RHS satisfies the equalizing conditions of the $\zeta$ on the LHS as a consequence of the properties of $\delta$ and $\eps$ given in Lemma \ref{lem: lambda-rho}. The form of the RHS of these equations makes it obvious that they satisfy the usual relations that a simplicial object in $\End\E$ should have. This proves that $\bfQ$ is a functor.

As for the monoidal structure $\nu^{m,n}:\bfQ_m\bfQ_n\to\bfQ_{m+n}$ the requirement that $\zeta$ be monoidal leaves only
one possibility,
\begin{equation} \label{def of nu}
\zeta^{m+n}\ci\nu^{m,n}=Q^m\zeta^n\ci\zeta^m\bfQ_n\equiv\zeta^m Q^n\ci\bfQ_m\zeta^n\,,
\end{equation}
which exists by the equalizing properties of the RHS. Since $\bfQ_0=\one_\E$, we can take $\iota$ to be the identity natural transformation $\one_\E\to\one_\E$, provided we also choose $\zeta^0$ to be the identity. Then the monoidality conditions on $\zeta$ are built in the definition of $\nu$ and $\iota$ and the monoidality constraints on $\nu$ and $\iota$ boil down to
\begin{align*}
\nu^{l+m,n}\ci\nu^{l,m}\bfQ_n&=\nu^{l,m+n}\ci\bfQ_l\nu^{m,n}\\
\nu^{0,n}&=\bfQ_n\\
\nu^{m,0}&=\bfQ_m.
\end{align*}
The last two follow from uniqueness of $\nu$ and the first can be shown by multiplying it with $\zeta^{l+m+n}$ and using 
(\ref{def of nu}).

Finally, we have to show naturality of $\nu$ (and of $\iota$). That is to say, we need a proof of
\begin{align*}
\bfQ_{f+g}\ci\nu^{m,n}&=\nu^{m',n'}\ci\bfQ_f\bfQ_g\\
&\forall f:m'\to m,\ g:n'\to n\ \text{in }\Delta\,.
\end{align*}
It suffices to prove this for $f$ and $g$ being elementary functions, that is to say, to prove
\[
\delta^{m+n}_i\ci\nu^{m,n}=\left\{
\begin{aligned} 
\nu^{m+1,n}\ci\delta^m_i\bfQ_n&\quad\text{if } i<m\\ \nu^{m,n+1}\ci\bfQ_m\delta^n_{i-m}&\quad\text{if }i\geq m 
\end{aligned}
\right.
\]
and 
\[
\eps^{m+n}_i\ci\nu^{m,n}=\left\{
\begin{aligned} 
\nu^{m-1,n}\ci\eps^m_i\bfQ_n&\quad\text{if } i<m\\ \nu^{m,n-1}\ci\bfQ_m\eps^n_{i-m}&\quad\text{if }i\geq m \,.
\end{aligned}
\right.
\]
Multiplying the first with $\zeta^{m+n+1}$ and the second with $\zeta^{m+n-1}$ they can be easily verified using the defining
relations (\ref{def of delta^n_i}) and (\ref{def of eps^n_i}).
\end{proof}

\begin{thm} \label{thm: equivalence of lax comodule cats}
The functor $\hat{\phi}:\E^\bfQ\to\M^Q$ induced by the lax comonad morphism 
$\bra\phi,\zeta\ket:\bra\E,\bfQ\ket\to\bra \M,Q\ket$ 
of the above Proposition is an equivalence of categories.
\end{thm}
\begin{proof}
$\hat{\phi}$ is the lift of the faithful $\phi$ along the Eilenberg-Moore forgetful functors,
\[
\begin{CD}
\E^\bfQ@>\hat{\phi}>>\M^Q\\
@VVV @VVV\\
\E@>\phi>>\M
\end{CD}
\]
therefore it is faithful, too. For an arrow $t\in\M^Q(\hat{\phi}\bra M,\alpha\ket,\hat{\phi}\bra N,\beta\ket)$ we have
\[
\begin{CD}
\phi M@>\phi\alpha_n>>\phi\bfQ_n M@>\zeta^n_M>>Q^n\phi M\\
@V{t}VV @. @VV{Q^nt}V\\
\phi N@>\phi\beta_n>>\phi\bfQ_nN@>\zeta^n_N>>Q^n\phi N
\end{CD}
\]
therefore by Lemma \ref{lem: forget to E-mod} $t=\phi\tau$ for a unique $\tau\in\E(M,N)$. This allows to insert the arrow $\phi\bfQ_n\tau$ in the middle of the diagram so that the right square is commutative. But $\zeta^n_N$ being monic implies
commutativity of the left square, so $\tau$ lifts to an arrow in $\E^\bfQ(\bra M,\alpha\ket,\bra N,\beta\ket)$. This proves that $\hat{\phi}$ is full.
Finally we show that $\hat{\phi}$ is eso, in fact surjective on objects. Let $\bra M,\alpha\ket\in\M^Q$. Then by Lemma 
\ref{lem: central coactions quotient actions} there is an $\bra \hat{M},\hat{\alpha}\ket\in\E^\bfQ$ such that
\[
\left(M\longrarr{\phi\hat{\alpha}_n}\phi\bfQ_n \hat{M}\equalizer{\zeta^n_M}Q^nM\right)=\alpha_n\equiv Q^{n-1}\alpha\ci\dots Q\alpha\ci\alpha
\]
i.e., such that $\hat\phi\bra\hat M,\hat\alpha\ket=\bra M,\alpha\ket$. Thus $\hat\phi$ is eso.
\end{proof}

\begin{rmk}
There is a lift of the distributive law $\chi:TQ\to QT$ of Lemma  \ref{lem: T Q chi} to a \textit{lax distributive law}
$\psi_n: T_q\bfQ_n\to\bfQ_nT_q$ provided we consider $T$, $Q$, $T_q$ and $\bfQ_n$ as endofunctors on the category
$\M^{(2)}$ of (2,1)-type $E$-objects, which is the category of $R^\op\ot R$-$R^\op\ot R$-bimodules in case $\M=\Ab_R$. 
Of course, $\M$ has to have limits and colimits and $R\smp\under$ has to preserve coequalizers in order for $T_q$ to be a monad
and $\kappa$ a monad morphism. More precisely, $T_q$ on $\M^{(2)}$ is defined as $T_q$ on $\M^{(1)}=\E$ by considering
$M\in\M^{(2)}$ as a $(1,1)$-type $E$-object in $\E$ via $\rho_1$ and $\lambda_2$  and $\bfQ_n$ on $\M^{(2)}$ is defined
as $\bfQ_n$ on $\M^{(1)}$ by considering $M\in\M^{(2)}$ as a $(1,1)$-type $E$-object in $\E$ via $\lambda_1$ and $\rho_1$.
Let  $\chi^n:=Q^{n-1}\chi\ci \dots\ci Q\chi Q^{n-2}\ci\chi Q^{n-1}$ and consider the diagram
\[
\parbox{200pt}{
\begin{picture}(200, 150)(0,-10)
\put(50,0){$Q^nT$}
\put(75,2){\vector(1,0){40}} \put(87,8){$\ssst Q^n\kappa$}
\put(120,0){$Q^nT_q$}
   \put(10,50){\vector(1,-1){40}} \put(16,25){$\ssst \chi^n$}
   \dashline{3}(80,50)(120,10) \put(120,10){\vector(1,-1){0}} \put (102,33){$\ssst \vartheta^n$}
   \put(175,50){\vector(-1,-1){40}} \put(160,25){$\ssst\zeta^nT_q$}
\put(0,60){$TQ^n$}
\put(25,62){\vector(1,0){40}} \put(42,68){$\ssst \kappa Q^n$}
\put(70,60){$T_qQ^n$}
\put(170,60){$\bfQ_n T_q$}
   \put(55,112){\vector(-1,-1){40}} \put(20,95){$\ssst T\zeta^n$}
   \put(120,112){\vector(-1,-1){40}} \put(103,85){$\ssst T_q\zeta^n$}
   \dashline{3}(133,112)(173,72) \put(173,72){\vector(1,-1){0}} \put(157,93){$\ssst\psi^n$}
\put(50,120){$T\bfQ_n$}
\put(75,122){\vector(1,0){40}} \put(90,128){$\ssst \kappa \bfQ_n$} 
\put(120,120){$T_q\bfQ_n$}
\end{picture}
}
\]
in which $\kappa Q^n$ is a coequalizer which defines $\vartheta^n$ and $\zeta^n T_q$ is an equalizer which defines $\psi^n$.
Then one obtains the distributive laws
\begin{align*}
\bfQ_n\mu^q\ci\psi^n T_q\ci T_q\psi^n&=\psi^n\ci\mu^q\bfQ_n\\
\psi^n\ci\eta^q\bfQ_n&=\bfQ_n\eta^q\\
\psi^{n+1}\ci T_q\delta^n_i&=\delta^n_iT_q\ci\psi^n&\\
\psi^{n-1}\ci T_q\eps^n_i&=\eps^n_i T_q\ci\psi^n
\end{align*}
as a consequence of (\ref{SMC14}), (\ref{SMC16}), (\ref{SMC15}) and (\ref{SMC17}), respectively. 
While the last two express only naturality of $\psi$ the first two contain the monad data $\bra T_q,\mu^q,\eta^q\ket$. The difference disappears, however, if we introduce the lax monad $\bfT$ as a cosimplicial object $\Delta\to\End\M^{(2)}$ by
\begin{align*}
\bfT_m&:=T_q^m\,,\qquad m\geq 0\\
\bfT_{i+(2\to 1)+j}&:=T_q^i\mu^q T_q^j\,,\qquad i,j\geq 0\\
\bfT_{i+(1\to 0)+j}&:=T_q^i\eps^q T_q^j\,,\qquad i,j\geq 0\,.
\end{align*}
Then the lax distributive law becomes deceptively simple, just a natural transformation
\[
\bfT\bfQ\longrarr{\psi}\bfQ\bfT \ :\ \Delta\x\Delta^\op\rightarrow \End\M^{(2)}\,.
\]
Note that $\bfT\bfQ$ and $\bfQ\bfT$ are not the composite of two functors as in $TQ\rarr{\chi}QT$, rather the monoidal product on their common target category: $\Delta\x\Delta^\op\rarr{\bfT\x\bfQ}\End\M^{(2)}\x\End\M^{(2)}\to\End\M^{(2)}$.
All information on the compatibility of $\psi^{m,n}$ with $\mu^q$, $\eta^q$, $\delta^n_i$, $\eps^n_i$ seems to be comprised in 
the naturality of $\psi^{m,n}$ in $m\in\Delta$ and $n\in\Delta^\op$. However, $\psi^{m,n}$ also satisfies some `monoidality' relations in $m$ and $n$ separately which are automatic in this example and which ought to belong to the axioms 
of a lax distributive law for general lax monad $\bfT$ and lax comonad $\bfQ$. 
\end{rmk}

In the rest of the paper we study the problem of how and when (ordinary) monoidal structures on the category $\E$ of $E$-objects
will lead to monoidality of the Eilenberg-Moore categories $\E^\bfQ$ or $\E_{T_q}$ with a strong monoidal forgetful functor to $\E$.

\section{Bi(co)monad induced structures} \label{s: O-ind}

In Section \ref{s: bgd} we have seen how right $R$-bialgebroids induce right-monoidal structures on the category $\Ab_R$ of right $R$-modules. Since bialgebroids correspond to bimonads, i.e., opmonoidal monads, on $\E=\,_R\Ab_R$ \cite{Sz: EM}, it is natural
to look for generalizations that produce right-monoidal categories from bimonads. 

Let $\bra\E,\ot,R,\asso,\luni^{-1},\runi\ket$ be a monoidal category. Then a bimonad, more precisely a $\ot$-bimonad, 
$\bra O,\omega,\iota\ket$ 
consists of an endofunctor $O$ on $\E$ together with an opmonoidal structure $O^{M,N}:O(M\ot N)\to OM\ot ON$, $O^0:OR\to R$ and  natural transformations $\omega:OO\to O$ and $\iota:\E\to O$ satisfying the monad axioms 
(not involving the $\ot$-structure) and the opmonoidality axioms
\begin{align}
\label{opmon1}
\asso_{OL,OM,ON}\ci(OL\ot O^{M,N})\ci O^{L,M\ot N}&=(O^{L,M}\ot ON)\ci O^{L\ot M,N}\ci O\asso_{L,M,N}\\
\label{opmon2}
(O^0\ot ON)\ci O^{R,N}\ci O\luni^{-1}_N&=\luni^{-1}_{ON}\\\
\label{opmon3}
\runi_{OM}\ci(OM\ot O^0)\ci O^{M,R}&=O\runi_M\\
\label{opmon4}
(\omega_M\ot\omega_N)\ci O^{OM,ON}\ci OO^{M,N}&=O^{M,N}\ci\omega_{M\ot N}\\
\label{opmon5}
O^0\ci\omega_R&=O^0\ci OO^0\\
\label{opmon6}
O^{M,N}\ci\iota_{M\ot N}&=\iota_M\ot\iota_N\\
\label{opmon7}
O^0\ci\iota_R&=R\,.
\end{align}
We have written them using only $\asso$, $\luni^{-1}$ and $\runi$ but never their inverses. This admits to speak about
opmonoidal monads in right-monoidal categories. Such right-opmonoidal monads are not really new, they are just the monads in
the 2-category $\RightOpmonCat$. Indeed, relations (\ref{opmon1}-\ref{opmon3}) say exactly that $O$ is a 1-cell and 
relations (\ref{opmon4}-\ref{opmon7}) say that $\omega$ and $\iota$ are 2-cells of this 2-category.

The so-called fusion operator \cite{BLV} associated to a bimonad 
$\bra O,\omega,\iota\ket$ is the natural transformation
\begin{equation} \label{H(O)}
h_{M,N}:=(OM\ot\omega_N)\ci O^{M,ON}\ :\ O(M\ot ON)\to OM\ot ON\,.
\end{equation}
Given a fusion operator we can recover the opmonoidal structure by
\begin{equation} \label{O(H)}
O^{M,N}=h_{M,N}\ci O(M\ot\iota_N).
\end{equation}
The next result is essentially \cite[Proposition 2.6]{BLV} of Bruguieres, Lack and Virelizier although some of the output is turned into
input. But the main difference is the observation that the statement is valid also when $\ot$ is a skew-monoidal product.
\begin{pro} \label{pro: fusion}
Let $\bra\E,\ot,R,\asso,\luni^{-1},\runi\ket$ be a right-monoidal category and $\bra O,\omega,\iota\ket$ be a monad on $\E$.
Then opmonoidal structures on $O$, i.e., $O^{M,N}$, $O^0$ satisfying (\ref{opmon1}-\ref{opmon7}), are in bijection with
data consisting of a natural transformation $h_{M,N}:O(M\ot ON)\to OM\ot ON$ and the same $O^0$ satisfying
the following relations:
\begin{align}
\label{H0}
(OM\ot\omega_N)\ci h_{M,ON}&=h_{M,N}\ci O(M\ot\omega_N)\\
\label{H1}
(h_{L,M}\ot ON)\ci h_{L\ot OM,N}\ci O\asso_{L,OM,ON}&\ci O(L\ot h_{M,N})=\\
\notag
=\asso_{OL,OM,ON}&\ci(OL\ot h_{M,N})\ci h_{L,M\ot ON}\\
\label{H2}
h_{M,N}\ci\iota_{M\ot ON}&=\iota_M\ot ON\\
\label{H3}
(O^0\ot ON)\ci h_{R,N}\ci O\luni^{-1}_{ON}&=\luni^{-1}_{ON}\ci\omega_N\\
\label{H4}
\runi_{OM}\ci(OM\ot O^0)\ci h_{M,R}&=O\runi_{M}\ci O(M\ot O^0)\\
\label{H5}
(\omega_M\ot ON)\ci h_{OM,N}\ci Oh_{M,N}&=h_{M,N}\ci \omega_{M\ot ON}\\
\label{H6}
O^0\ci \iota_R&=R\,.
\end{align}
The bijection is given by equations (\ref{H(O)}) and (\ref{O(H)}).
\end{pro}
\begin{proof}
Assume that an opmonoidal structure $O^{M,N}$, $O^0$ is given and $h$ is defined by (\ref{H(O)}). 
Then (\ref{H0}) can be shown using associativity of the monad multiplication $\omega$,
\begin{align*}
(OM\ot\omega_N)\ci h_{M,ON}&=(OM\ot\omega_N)\ci(OM\ot \omega_{ON})\ci O^{M,O^2N}=\\
&=(OM\ot\omega_N)\ci(OM\ot O\omega_N)\ci O^{M,O^2N}=\\
&=(OM\ot\omega_N)\ci O^{M,ON}\ci O(M\ot\omega_N)=h_{M,N}\ci O(M\ot\omega_N)\,.
\end{align*}
The proof of the associativity law (\ref{H1}) is a bit longer:
\begin{align*}
&(h_{L,M}\ot ON)\ci h_{L\ot  OM,N}\ci O\asso_{L,OM,ON}\ci O(L\ot h_{M,N})=\\
&=((OL\ot\omega_M)\ot  ON)\ci(O^{L,OM}\ot ON)\ci(O(L\ot OM)\ot\omega_N)\ci O^{L\ot OM,ON}\ci O\asso_{L,OM,ON}\\
&\qquad\ci O(L\ot(OM\ot\omega_N))\ci O(L\ot O^{M,ON})=\\
&=((OL\ot\omega_M)\ot\omega_N)\ci((OL\ot O^2M)\ot O\omega_N)\ci(O^{L,OM}\ot O^3N)\ci O^{L\ot OM,O^2N}\\
&\qquad\ci O\asso_{L,OM,O^2N}\ci O(L\ot O^{M,ON})=\\
&\eqby{opmon1}((OL\ot\omega_M)\ot\omega_N)\ci(((OL\ot O^2M)\ot O\omega_N)\ci\asso_{OL,O^2M,O^3N}\\
&\qquad\ci(OL\ot O^{OM,O^2N})\ci O^{L,OM\ot O^2N}\ci O(L\ot O^{M,ON})=\\
&=\asso_{OL,OM,ON}\ci(OL\ot(\omega_M\ot\omega_N))\ci(OL\ot(O^2M\ot O\omega_N))\\
&\qquad\ci (OL\ot O^{OM,O^2N})\ci (OL\ot OO^{M,ON})\ci O^{L,O(M\ot ON)}=\\
&=\asso_{OL,OM,ON}\ci(OL\ot(OM\ot\omega_N))\ci(OL\ot\left[(\omega_M\ot\omega_{ON})\ci O^{OM,O^2N}\ci OO^{M,ON}\right])\\
&\qquad\ci O^{L,O(M\ot ON)}=\\
&\eqby{opmon4}\asso_{OL,OM,ON}\ci(OL\ot h_{M,N})\ci h_{L,M\ot ON}\,.
\end{align*}
As for the remaining relations we proceed as follows:
\begin{align*}
h_{M,N}\ci\iota_{M\ot ON}&=(OM\ot\omega_N)\ci O^{M,ON}\ci\iota_{M\ot ON}=\\
&\eqby{opmon6}(OM\ot\omega_N)\ci(\iota_M\ot\iota_{ON})=\iota_M\ot ON\,,
\end{align*}
\begin{align*}
(O^0\ot ON)\ci h_{R,N}\ci O\luni^{-1}_{ON}&=(R\ot\omega_N)\ci(O^0\ot O^2N)\ci O^{R,ON}\ci O\luni^{-1}_{ON}=\\
&\eqby{opmon2}(R\ot\omega_N)\ci\luni^{-1}_{O^2N}=\luni^{-1}_{ON}\ci\omega_N\,,
\end{align*}
\begin{align*}
\runi_{OM}\ci(OM\ot O^0)\ci h_{M,R}&=\runi_{OM}\ci(OM\ot O^0)\ci (OM\ot\omega_R)\ci O^{M,OR}=\\
&\eqby{opmon5}\runi_{OM}\ci(OM\ot O^0)\ci (OM\ot OO^0)\ci O^{M,OR}=\\
&=\runi_{OM}\ci(OM\ot O^0)\ci O^{M,R}\ci O(M\ot O^0)=\\
&\eqby{opmon3}O\runi_{M}\ci O(M\ot O^0)
\end{align*}
and finally (\ref{H5}) follows from (\ref{opmon4}) easily.

Now assume that a fusion operator $h$ is given, together with $O^0$, and define $O^{M,N}$ by (\ref{O(H)}). 
First, (\ref{opmon6}) follows easily from (\ref{H2}). Then associativity relation (\ref{opmon1}) can be shown by means of 
(\ref{H1}) and (\ref{opmon6}):
\begin{align*}
&\asso_{OL,OM,ON}\ci(OL\ot O^{M,N})\ci O^{L,M\ot N}=\\
&=\asso_{OL,OM,ON}\ci(OL\ot h_{M,N})\ci(OL\ot O(M\ot \iota_N))\ci h_{L,M\ot N}\ci O(L\ot\iota_{M\ot N})=\\
&=\asso_{OL,OM,ON}\ci(OL\ot h_{M,N})\ci h_{L,M\ot ON}\ci O(L\ot O(M\ot\iota_N))\ci O(L\ot\iota_{M\ot N})=\\
&\eqby{H1}(h_{L,M}\ot ON)\ci h_{L\ot OM,N}\ci O\asso_{L,OM,ON}\ci O(L\ot h_{M,N})\ci O(L\ot O(M\ot\iota_N))
\ci O(L\ot\iota_{M\ot N})\\
&\eqby{opmon6}(h_{L,M}\ot ON)\ci h_{L\ot OM,N}\ci O\asso_{L,OM,ON}\ci O(L\ot (\iota_M\ot\iota_N))=\\
&=(h_{L,M}\ot ON)\ci(O(L\ot\iota_M)\ot ON)\ci h_{L\ot M,N}\ci O((L\ot M)\ot \iota_N)\ci O\asso_{L,M,N}=\\
&=(O^{L,M}\ot ON)\ci O^{L\ot M,N}\ci O\asso_{L,M,N}\,.
\end{align*}
Equation (\ref{opmon2}) is a simple consequence of (\ref{H3}) if we compose the latter with $O\iota_N$.
Similarly, (\ref{opmon3}) follows from (\ref{H4}) and (\ref{H6}). 
For proving (\ref{opmon4}) we need relation (\ref{H5}) and the calculation
\begin{align*}
&{}(\omega_M\ot\omega_N)\ci h_{OM,ON}\ci O(OM\ot\iota_{ON})\ci Oh_{M,N}\ci O^2(M\ot\iota_N)=\\
&=(\omega_M\ot ON)\ci h_{OM,N}\ci Oh_{M,N}\ci O^2(M\ot\iota_N)=\\
&\eqby{H5}h_{M,N}\ci \omega_{M\ot ON}\ci O^2(M\ot\iota_N)=\\
&=h_{M,N}\ci O(M\ot\iota_N)\ci\iota_{M\ot N}=O^{M,N}\ci \iota_{M\ot N}\,.
\end{align*}
Finally, (\ref{opmon5}) is the consequence of (\ref{H3}) and (\ref{H4}),
\begin{align*}
O^0\ci\omega_R&=\runi_R\ci\luni^{-1}_R\ci O^0\ci\omega_R=\runi_R\ci(R\ot O^0)\ci\luni^{-1}_{OR}\ci\omega_R=\\
&\eqby{H3} \runi_R\ci(O^0\ot O^0)\ci h_{R,R}\ci O\luni^{-1}_{OR}=O^0\ci\runi_{OR}\ci (OR\ot O^0)\ci h_{R,R}\ci
O\luni^{-1}_{OR}=\\
&\eqby{H4}O^0\ci O\runi_R\ci O(R\ot O^0)\ci O\luni^{-1}_{OR}= O^0\ci O\runi_R\ci O\luni^{-1}_{OR}\ci OO^0=\\
&=O^0\ci OO^0\,.
\end{align*}
This finishes the proof that $O$ is opmonoidal.

It remains to verify that (\ref{H(O)}) and (\ref{O(H)}) define a bijection between fusion operators and opmonoidal structures.
While the composite mapping $O^{M,N}\mapsto h_{M,N}\mapsto O^{M,N}$ is the identity for whatever $O^{M,N}$, the composite
$h_{M,N}\mapsto O^{M,N}\mapsto h_{M,N}$ becomes the identity after using (\ref{H0}).
\end{proof}

\begin{pro} \label{pro: O-ind}
Let $\bra O,\omega,\iota\ket$ be a bimonad on the (right-)monoidal category $\bra\E,\ot,R,\asso,\luni^{-1},\runi\ket$. 
Then there is a right-monoidal structure on $\E$ given by 
\begin{align}
M\odot N&:=M\ot ON\\
\label{dotgamma}
\dot\gamma_{L,M,N}&:=\asso_{L,OM,ON}\ci(L\ot(OM\ot\omega_N))\ci(L\ot O^{M,ON})\\
\label{doteta}
\dot\eta_M&:=\luni^{-1}_{OM}\ci\iota_M\\
\label{doteps}
\dot\eps_M&:=\runi_M\ci(M\ot O^0)\,.
\end{align}
The unit $\luni^{-1}$ of the $\ot$-structure gives rise to a monad morphism $\luni^{-1}_{ON}:ON\to \dot TN$ from $O$ to the canonical monad $\dot T=R\odot\under$ of the $\odot$-structure.
\end{pro}
\begin{proof}
By Proposition \ref{pro: fusion} the monad $O$ is supplied with a fusion operator $h$. Since the associator $\dot\gamma$ is
essentially given by the fusion operator, the pentagon equation (\ref{SMC1}) for the $\odot$ product is a consequence of (\ref{H1}) and of the pentagon equation for $\ot$,
\begin{align*}
&(\dot\gamma_{K,L,M}\odot N)\ci\dot\gamma_{K,L\odot M,N}\ci(K\odot\dot\gamma_{L,M,N})=\\
&=(\asso_{K,OL,OM}\ot ON)\ci ((K\ot h_{L,M})\ot ON)\ci\asso_{K,O(L\ot OM),ON}\ci(K\ot h_{L\ot OM,N})\\
&\qquad\ci(K\ot O\asso_{L,OM,ON})\ci(K\ot O(L\ot h_{M,N}))=\\
&=(\asso_{K,OL,OM}\ot ON)\ci\asso_{K,OL\ot OM,ON}\\
&\qquad\ci\left(K\ot\left[(h_{L,M}\ot ON)\ci h_{L\ot OM,N}\ci O\asso_{L,OM,ON}\ci O(L\ot h_{M,N})\right]\right)=\\
&\eqby{H1}(\asso_{K,OL,OM}\ot ON)\ci\asso_{K,OL\ot OM,ON}\ci(K\ot\asso_{OL,OM,ON})\\
&\qquad\ci(K\ot(OL\ot h_{M,N}))\ci(K\ot h_{L,M\ot ON})=\\
&=\asso_{K\ot OL,OM,ON}\ci\asso_{K,OL,OM\ot ON}\ci(K\ot(OL\ot h_{M,N}))\ci(K\ot h_{L,M\ot ON})=\\
&=\dot\gamma_{K\odot L,M,N}\ci\dot\gamma_{K,L,M\odot N}\,.
\end{align*}
The unit-triangle (\ref{SMC2}) for $\odot$ follows from (\ref{H2}) and from the unit triangle for $\ot$,
\begin{align*}
\dot\gamma_{R,M,N}\ci\dot\eta_{M\odot N}&=\asso_{R,OM,ON}\ci(R\ot h_{M,N})\ci(R\ot\iota_{M\ot ON})\ci\luni^{-1}_{M\ot ON}=\\
&\eqby{H2}\asso_{R,OM,ON}\ci(R\ot(\iota_M\ot ON))\ci\luni^{-1}_{M\ot ON}=\\
&=((R\ot\iota_M)\ot ON)\ci(\luni^{-1}_M\ot ON)=\dot\eta_M\odot N\,.
\end{align*}
The counit-triangle (\ref{SMC3}) for $\odot$ follows from the counit triangle for $\ot$ and from (\ref{H4}),
\begin{align*}
\eps_{M\odot N}\ci\dot\gamma_{M,N,R}&=\runi_{M\ot ON}\ci((M\ot ON)\ot O^0)\ci\asso_{M,ON,OR}\ci(M\ot h_{N,R})=\\
&=(M\ot\runi_{ON})\ci(M\ot (ON\ot O^0))\ci(M\ot h_{N,R})=\\
&\eqby{H4}(M\ot O\runi_N)\ci (M\ot O(N\ot O^0))=M\odot\dot\eps_N\,.
\end{align*}
The mixed triangle (\ref{SMC4}) can be shown using (\ref{H3}) and then the analogous triangle for $\ot$:
\begin{align*}
&(\dot\eps_M\odot N)\ci\dot\gamma_{M,R,N}\ci(M\odot\dot\eta_N)=\\
&=((\runi_M\ci (M\ot O^0))\ot ON)\ci\asso_{M,OR,ON}\ci(M\ot h_{R,N})\ci(M\ot O\luni^{-1}_{ON}\ci O\iota_N)=\\
&=(\runi_M\ot ON)\ci \asso_{M,R,ON}\ci(M\ot (O^0\ot ON))\ci(M\ot h_{R,N})\ci(M\ot O\luni^{-1}_{ON})\ci(M\ot O\iota_N)=\\
&\eqby{H3}(\runi_M\ot ON)\ci \asso_{M,R,ON}\ci (M\ot\luni^{-1}_{ON})\ci(M\ot\omega_N)\ci(M\ot O\iota_N)=\\
&=(\runi_M\ot ON)\ci \asso_{M,R,ON}\ci (M\ot\luni^{-1}_{ON})=M\odot N\,.
\end{align*}
Finally, (\ref{SMC5}) for $\odot$ follows from (\ref{opmon7}) and from the analogous axiom for $\ot$,
\[
\dot\eps_R\ci\dot\eta_R=\runi_R\ci(R\ot O^0)\ci(R\ot\iota_R)\ci\luni^{-1}_R=\runi_R\ci\luni^{-1}_R=R\,.
\]
This finishes the proof that $\odot$ is a right-monoidal structure. The natural transformation $\luni^{-1}_{ON}$
 (together with the identity functor on $\E$) is a monad morphism $O\to\dot T$ if it satisfies the following two
conditions:
\begin{align} 
\label{monmor1}
\dot\mu_N\ci\luni^{-1}_{O\dot TN}\ci O\luni^{-1}_{ON}&=\luni^{-1}_{ON}\ci\omega_N\\
\label{monmor2}
\dot\eta_M&:=\luni^{-1}_{OM}\ci\iota_M\,.
\end{align}
The LHS of the first can be written as
\begin{align*}
&(\dot\eps_R\ot ON)\ci\dot\gamma_{R,R,N}\ci\luni^{-1}_{O(R\ot ON)}\ci O\luni^{-1}_{ON}=\\
&=(\runi_R\ot ON)\ci((R\ot O^0)\ot ON)\ci\asso_{R,OR,ON}\ci (R\ot h_{R,N})\ci(R\ot O\luni^{-1}_{ON})\ci\luni^{-1}_{O^2N}=\\
&\eqby{H3}(\runi_R\ot ON)\ci\asso_{R,R,ON}\ci(R\ot\luni^{-1}_{ON})\ci(R\ot\omega_N)\ci\luni^{-1}_{O^2N}=\\
&=\luni^{-1}_{ON}\ci\omega_N
\end{align*}
which is the RHS. The second condition is just the definition (\ref{doteta}) of $\dot\eta$, 
so $\luni^{-1}_{ON}$ is a monad morphism as claimed.
\end{proof}

\begin{defi} \label{def: repr}
The right-monoidal structures twist isomorphic (see Definition \ref{def: twist}) to ones arising from a bimonad w.r.t. some ordinary monoidal structure $\ot$ as in Proposition \ref{pro: O-ind} are called $\ot$-representable or representable by a $\ot$-bimonad.
\end{defi}

Passing to the reversed right-monoidal structures one obtains the notion of representability of left-monoidal categories by 
opmonoidal monads. Up to twist isomorphism they are given by
\begin{align*}
M\odot N&:=OM\ot N\\
\dot\gamma_{L,M,N}&:=\asso^{-1}_{OL,OM,N}\ci((\omega_L\ot OM)\ot N)\ci(O^{OL,M} \ot N)\\
\dot\eta_M&:=\runi^{-1}_{OM}\ci\iota_M\\
\dot\eps_M&:=\luni_M\ci(O^0\ot M)\,.
\end{align*}

Passing to the opposite category  opmonoidal monads become monoidal comonads and we obtain the notion of \textit{corepresentability}.
\begin{defi} \label{def: corep}
A right-monoidal category $\bra \M,\smp,R,\gamma,\eta,\eps\ket$ is corepresentable by a monoidal comonad 
$\bra C,C_2,C_0,\Delta,\epsilon\ket$ in a (left-) monoidal structure $\bra \M,\ot,R,\asso^{-1},\runi^{-1},\luni\ket$ when it is twist-isomorphic to the following right-monoidal structure:
\begin{align*}
M\odot N&:=N\ot CM\\
\dot\gamma_{L,M,N}&:=(N\ot C_{M,CL})\ci(N\ot(CM\ot\Delta_L))\ci\asso^{-1}_{N,CM,CL}\\
\dot\eta_M&:=(M\ot C_0)\ci\runi^{-1}_{M}\\
\dot\eps_M&:=\epsilon_M\ci\luni_{CM}\,.
\end{align*}
\end{defi}
It is left to the reader to write up what corepresentability means for left-monoidal categories.

\section{The representability theorem} \label{s: rep}

We wish to study the situation of a category $\E$ endowed with two right-monoidal structures $\bra\E,\smp,R,\gamma,\eta,\eps\ket$
and $\bra\E,\ot,R,\asso,\luni^{-1},\runi\ket$ with a common unit object $R$.
Later the second structure will be assumed to be an ordinary monoidal structure, this explains the notation, but for a good while
the unit $\luni^{-1}_M:M\to R\ot M$ is not assumed to be invertible, neither are $\asso_{L,M,N}$ and $\runi_M$.
We  shall briefly refer to them as the $\smp$-structure and the $\ot$-structure.

In order to relate this situation to that of earlier sections one may think $\E$ as the category of left $E$-objects in $\Ab_R$, i.e.,
$\E$ is the bimodule category $_R\Ab_R$ with $\ot$ the tensor product $\oR$. Then $\smp$ is the quotient $\smpq$ of a
right-monoidal structure on $\Ab_R$ as it was described in Proposition \ref{pro: gamma+}.

\begin{defi}
A tetrahedral homomorphism from the $\smp$-structure to the $\ot$-structure is a natural transformation
\[
\tet_{L,M,N}: \ L\ot (M\smp N)\ \to\ (L\ot M)\smp N
\]
satisfying the following axioms:
\begin{align}
\label{P*}
(\asso_{K,L,M}\smp N)\ci\tet_{K,L\ot M,N}\ci(K\ot\tet_{L,M,N})&=\tet_{K\ot L,M,N}\ci\asso_{K,L,M\smp N}\\
\label{P**}
(\tet_{K,L,M}\smp N)\ci\tet_{K,L\smp M,N}\ci(K\ot\gamma_{L,M,N})&=
\gamma_{K\ot L,M,N}\ci \tet_{K,L,M\smp N}\\
\label{tet-unit}
\tet_{R,M,N}\ci\luni^{-1}_{M\smp N}&=\luni^{-1}_M\smp N\\
\label{tet-counit}
\eps_{M\ot N}\ci\tet_{M,N,R}&=M\ot\eps_N\,.
\end{align}
A tetrahedral isomorphism is a tetrahedral homomorphism $\tet$ for which 
\begin{equation} \label{w(tet)}
w_{M,N}:=(\runi_M\smp N)\ci\tet_{M,R,N}\ :\ M\ot TN\to M\smp N
\end{equation}
is a natural isomorphism where $T=R\smp\under$.
\end{defi}

Axioms (\ref{P*}) and (\ref{P**}) are pentagons on the string of symbols $K\ot L\ot M\smp N$ and $K\ot L\smp M\smp N$,
respectively. Axioms (\ref{tet-unit}) and (\ref{tet-counit}) are analogous to the unit and counit axioms (\ref{SMC2}) and (\ref{SMC3}). The analogue of (\ref{SMC4}) is void since we have no distinguished arrow $M\ot N\to M\smp N$
to put on the right hand side, except the one on the left hand side.

The above axioms for $t$ can be recognized to be a fragment of the Cockett-Seely axioms for `linearly distributive categories' \cite{Cockett-Seely} although we do not assume either $\smp$ or $\ot$ to be monoidal structures. Our terminology 'tetrahedral" refers to the early 90s when A. Ocneanu used a tetrahedral calculus to formulate his `double-triangle algebras' 
\cite{Ocneanu, Petkova-Zuber}.

\begin{lem} \label{lem: tet --> w}
For $\tet$ a tetrahedral isomorphism from a $\smp$-structure to a $\ot$-structure we have the following results.
\begin{gather}
\label{w_R,N}
w_{R,N}\ci\luni^{-1}_{TN}=TN\\
\label{tet(w)}
\tet_{L,M,N}=w_{L\ot M,N}\ci\asso_{L,M,TN}\ci(L\ot w^{-1}_{M,N})\\
\label{P** spec}
(w_{L,M}\smp N)\ci w_{L\ot TM,N}\ci\asso_{L,TM,TN}\ci(L\ot w^{-1}_{TM,N})
\ci(L\ot \gamma_{R,M,N})=\gamma_{L,M,N}\ci w_{L,M\smp N}\\
\label{tet-counit spec}
\eps_M\ci w_{M,R}=\runi_M\ci(M\ot\eps_R)\,.
\end{gather}
\end{lem}
\begin{proof}
Setting $M=R$ in (\ref{tet-unit}) and multiplying it with $\runi_R\smp N$ we obtain 
$w_{R,N}\ci \luni^{-1}_{TN}=(\runi_R\smp N)\ci(\luni^{-1}_R\smp N)$ the RHS of which is the identity by axiom (\ref{SMC5}) for the $\ot$-structure. This proves (\ref{w_R,N}).

Set $(K,L,M,N)=(L,M,R,N)$ in the pentagon (\ref{P*}), multiply it with $\runi_{L\ot M}\smp N$ and use (\ref{SMC3}) for
the $\ot$. Then we obtain
\[
((L\ot\runi_M)\smp N)\ci \tet_{L,M\ot R,N}\ci(L\ot t_{M,R,N})=w_{L\ot M,N}\ci\asso_{L,M,TN}\,.
\]
Using naturality of $\tet$ the LHS becomes $\tet_{L,M,N}\ci (L\ot w_{M,N})$ from which (\ref{tet(w)}) follows immediately.

Setting $(K,L,M,N)=(L,R,M,N)$ in (\ref{P**}) and then multiplying it with $(\runi_L\smp M)\smp N$ we obtain
\[
(w_{L,M}\smp N)\ci \tet_{L,TM,N}\ci(L\ot\gamma_{R,M,N}) =\gamma_{L,M,N}\ci w_{L,M\smp N}\,.
\]
Inserting here the expression (\ref{tet(w)}) we obtain the heptagon (\ref{P** spec}).

Setting $N=R$ in (\ref{tet-counit}), multiplying it with $\runi_M$ and then using naturality of $\eps$ on the LHS leads
to (\ref{tet-counit spec}).
\end{proof}

\begin{pro} \label{pro: tet-w}
Given right-monoidal structures $\ot$ and $\smp$ on the same category and with same unit object $R$
equations (\ref{w(tet)}) and (\ref{tet(w)}) provide a bijection between 
\begin{trivlist}
\item tetrahedral isomorphisms $\tet_{L,M,N}:L\ot(M\smp N)\to(L\ot M)\smp N$ 
\item and natural isomorphisms $w_{M,N}:M\ot TN\iso M\smp N$ 
satisfying (\ref{P** spec}) and (\ref{tet-counit spec}).
\end{trivlist}
\end{pro}
\begin{proof}
Given a tetrahedral isomorphism $\tet$ the natural isomorphism $w$ defined by (\ref{w(tet)}) satisfies (\ref{P** spec}) and (\ref{tet-counit spec}) by Lemma \ref{lem: tet --> w}.

Assume $w$ is a natural isomorphism satisfying (\ref{P** spec}) and (\ref{tet-counit spec}) and define the natural transformation $\tet$ by (\ref{tet(w)}). Then the pentagon (\ref{P*}) is a simple consequence of the pentagon for $\asso$ (and invertibility of $w$). But in order to prove the other pentagon (\ref{P**}) we need its special case (\ref{P** spec}). The LHS 
of (\ref{P**}) can be written as
\begin{align*}
\text{LHS}&=(w_{K\ot L,M}\smp N)\ci(\asso_{K,L,TM}\smp N)\ci((K\ot w^{-1}_{L,M})\smp N)\\
&\qquad\ci w_{K\ot(L\smp M),N}\ci\asso_{K,L\smp M,TN}\ci(K\ot w^{-1}_{L\smp M,N})\ci(K\ot\gamma_{L,M,N})=\\
&=(w_{K\ot L,M}\smp N)\ci w_{(K\ot L)\ot TM,N}\ci(\asso_{K,L,TM}\ot TN)\ci\asso_{K,L\ot TM,TN}\\
&\qquad\ci(K\ot w^{-1}_{L\ot TM,N})\ci(K\ot(w^{-1}_{L,M}\smp N))\ci (K\ot\gamma_{L,M,N})=\\
&\eqby{P** spec}(w_{K\ot L,M}\smp N)\ci w_{(K\ot L)\ot TM,N}\ci(\asso_{K,L,TM}\ot TN)\ci\asso_{K,L\ot TM,TN}\\
&\qquad(K\ot\asso_{L,TM,TN})\ci(K\ot(L\ot w^{-1}_{TM,N}))\ci(K\ot(L\ot\gamma_{R,M,N}))\ci(K\ot w^{-1}_{L,M\smp N})=\\
&=(w_{K\ot L,M}\smp N)\ci w_{(K\ot L)\ot TM,N}\ci\asso_{K\ot L,TM,TN}\ci\asso_{K,L,TM\ot TN}\\
&\qquad\ci(K\ot(L\ot w^{-1}_{TM,N}))\ci(K\ot(L\ot\gamma_{R,M,N}))\ci(K\ot w^{-1}_{L,M\smp N})=\\
&=(w_{K\ot L,M}\smp N)\ci w_{(K\ot L)\ot TM,N}\ci\asso_{K\ot L,TM,TN}\ci ((K\ot L)\ot w^{-1}_{TM,N})\\
&\qquad\ci((K\ot L)\ot\gamma_{R,M,N})\ci\asso_{K,L,T(M\smp N)}\ci(K\ot w^{-1}_{L,M\smp N})=\\
&\eqby{P** spec}\gamma_{K\ot L,M,N}\ci w_{K\ot L,M\smp N}\ci\asso_{K,L,T(M\smp N)}\ci(K\ot w^{-1}_{L,M\smp N})
\end{align*}
which is exactly the RHS.
In order to prove (\ref{tet-unit}) insert $L=R$ in the definition (\ref{tet(w)}) of $\tet$ and multiply it with $\luni^{-1}_{M\smp N}$.
\begin{align*}
\tet_{R,M,N}\ci\luni^{-1}_{M\smp N}&=w_{R\ot M,N}\ci\asso_{R,M,TN}\ci\luni^{-1}_{M\ot TN}\ci w^{-1}_{M,N}=\\
&=w_{R\ot M,N}\ci(\luni^{-1}_{M}\ot TN)\ci w^{-1}_{M,N}=\luni^{-1}_M\smp N
\end{align*}
where we used (\ref{SMC2}) for $\ot$. Axiom (\ref{tet-counit}) in turn can be proven by using 
(\ref{tet-counit spec})
and (\ref{SMC3}) for $\ot$:
\begin{align*}
\eps_{M\ot N}\ci\tet_{M,N,R}&=\eps_{M\ot N}\ci w_{M\ot N,R}\ci\asso_{M,N,TR}\ci(M\ot w^{-1}_{N,R})=\\
&\eqby{tet-counit spec})\runi_{M\ot N}\ci((M\ot N)\ot\eps_R)\ci\asso_{M,N,TR}\ci(M\ot w^{-1}_{N,R})=\\
&=\runi_{M\ot N}\ci\asso_{M,N,R}\ci(M\ot ((N\ot\eps_R)\ci w^{-1}_{N,R}))=\\
&\eqby{SMC3}M\ot\left[\runi_N\ci(N\ot\eps_R)\ci w^{-1}_{N,R}\right] \eqby{tet-counit spec}M\ot\eps_N\,.
\end{align*}
This finishes the proof that $\tet$ is a tetrahedral homomorphism. That it is also a tetrahedral isomorphism will be a 
consequence of that the composite map $w\mapsto\tet\mapsto w$ is the identity. 
Indeed, it maps $w$ to
\begin{align*}
(\runi_M\smp N)\ci w_{M\ot R,N}\ci\asso_{M,R,TN}\ci(M\ot w^{-1}_{R,N})&=w_{M,N}\ci(\runi_M\ot TN)\ci\asso_{M,R,TN}\ci (M\ot\luni^{-1}_{TN})=\\
&=w_{M,N}
\end{align*}
by (\ref{w_R,N}) and by the (\ref{SMC4}) axiom for $\ot$. 
That $\tet\mapsto w\mapsto \tet$ is also the identity has been already proven in Lemma \ref{lem: tet --> w} when we 
verified (\ref{tet(w)}).
\end{proof}

Note that in case of tetrahedral isomorphisms axiom (\ref{tet-unit}) is redundant, it follows from (\ref{P*}) alone.
Indeed, in Lemma \ref{lem: tet --> w} (\ref{tet(w)}) was a consequence of only (\ref{P*}) and in the proof of Proposition 
\ref{pro: tet-w} we derived axiom (\ref{tet-unit}) using only (\ref{tet(w)}).

Having a natural isomorphism $w$ as in Proposition \ref{pro: tet-w} we can define what looks like an opmonoidal structure for the canonical monad $T$, namely
\begin{align}
\label{T(w)}
T^{M,N}&:=w^{-1}_{TM,N}\ci\gamma_{R,M,N}\ci Tw_{M,N}\ci T(M\ot\eta_N)\ :\ T(M\ot N)\to TM\ot TN\\
\label{T0(w)}
T^0&:=\eps_R\ :\ TR\to R\,.
\end{align}
In order to prove that they make the monad $\bra T,\mu,\eta\ket$ opmonoidal,
we use the technology of fusion operators. In contrast to Section \ref{s: O-ind}, however, we need $h$ to be expressed
in terms of $w$. Comparing (\ref{T(w)}) with (\ref{H(O)}) the conjecture is that  
\begin{equation} \label{H(w)}
h_{M,N}:=w^{-1}_{TM,N}\ci\gamma_{R,M,N}\ci Tw_{M,N}\ :\ T(M\ot TN)\to TM\ot TN
\end{equation}
is a fusion operator.

\begin{lem} \label{lem: H}
Let the natural isomorphism $w$ satisfy (\ref{P** spec}) and (\ref{tet-counit spec}). Then (\ref{H(w)}), together with $T^0=\eps_R$, 
is a fusion operator for the monad $\bra T,\mu,\eta\ket$, i.e., it satisfies equations (\ref{H0}-\ref{H6}) with 
$O,\omega,\iota,O^0$ replaced by $T,\mu,\eta,T^0$, respectively.
\end{lem}
\begin{proof}
First we prove (\ref{H1}) by unpacking it by means of (\ref{H(w)}) and then using (\ref{P** spec}) twice:
\begin{align*}
&(h_{L,M}\ot TN)\ci h_{L\ot TM,N}\ci T\asso_{L,TM,TN}\ci T(L\ot h_{M,N})\\
&=(w^{-1}_{TL,M}\ot TN)\ci(\gamma_{R,L,M}\ot TN)\ci (Tw_{L,M}\ot TN)\ci w^{-1}_{T(L\ot TM),N}\ci\gamma_{R,L\ot TM,N}\\
&\qquad\ci Tw_{L\ot TM,N}\ci T\asso_{L,TM,TN}\ci T(L\ot w^{-1}_{TM,N})\ci T(L\ot\gamma_{R,M,N})\ci T(L\ot Tw_{M,N})=\\
&=(w^{-1}_{TL,M}\ot TN)\ci w^{-1}_{TL\smp M,N}\ci(\gamma_{R,L,M}\smp N)\ci\gamma_{R,L\smp M,N}\\
&\qquad \ci T\left[(w_{L,M}\smp N)\ci w_{L\ot TM,N}\ci \asso_{L,TM,TN}\ci (L\ot w^{-1}_{TM,N})\ci (L\ot\gamma_{R,M,N})\ci (L\ot Tw_{M,N})\right]=\\
&\eqby{P** spec}(w^{-1}_{TL,M}\ot TN)\ci w^{-1}_{TL\smp M,N}\ci(\gamma_{R,L,M}\smp N)\ci\gamma_{R,L\smp M,N}\ci T\gamma_{L,M,N}\\
&\qquad\ci Tw_{L,M\smp N}\ci T(L\ot Tw_{M,N})=\\
&\eqby{SMC1}(w^{-1}_{TL,M}\ot TN)\ci w^{-1}_{TL\smp M,N}\ci\gamma_{TL,M,N}\ci\gamma_{R,L,M\smp N}
\ci Tw_{L,M\smp N}\ci T(L\ot Tw_{M,N})=\\
&=w^{-1}_{TL\ot TM,N}\ci(w^{-1}_{TL,M}\smp N)\ci\gamma_{TL,M,N}\ci w_{TL,M\smp N}\ci h_{L,M\smp N}\ci T(L\ot Tw_{M,N})=\\
&\eqby{P** spec}\asso_{TL,TM,TN}\ci(TL\ot w^{-1}_{TM,N})\ci(TL\ot \gamma_{R,M,N})\ci h_{L,M\smp N}\ci T(L\ot Tw_{M,N})=\\
&=\asso_{TL,TM,TN}\ci(TL\ot h_{M,N})\ci h_{L,M\ot TN}
\end{align*}
Equations (\ref{H2}), (\ref{H3}) and (\ref{H4}) can be shown as follows:
\begin{align*}
h_{M,N}\ci\eta_{M\ot TN}&=w^{-1}_{TM,N}\ci\gamma_{R,M,N}\ci Tw_{M,N}\ci\eta_{M\ot TN}=\\
&=w^{-1}_{TM,N}\ci\gamma_{R,M,N}\ci\eta_{M\smp N}\ci w_{M,N}=\\
&\eqby{SMC2}w^{-1}_{TM,N}\ci(\eta_M\smp N)\ci w_{M,N}=\eta_M\ot TN\,.
\end{align*}
\begin{align*}
(T^0\ot TN)\ci h_{R,N}\ci T\luni^{-1}_{TN}&=(\eps_R\ot TN)\ci w^{-1}_{TR,N}\ci\gamma_{R,R,N}\ci Tw_{R,N}\ci T\luni^{-1}_{TN}=\\
&\eqby{w_R,N} w^{-1}_{R,N}\ci (\eps_R\smp N)\ci\gamma_{R,R,N}=\\
&\eqby{w_R,N} \luni^{-1}_{TN}\ci\mu_N\,.
\end{align*}
\begin{align*}
\runi_{TM}\ci(TM\ot T^0)\ci h_{M,R}&=\runi_{TM}\ci(TM\ot\eps_R)\ci w^{-1}_{TM,R}\ci\gamma_{R,M,R}\ci Tw_{M,R}=\\
&\eqby{tet-counit spec}\eps_{TM}\ci\gamma_{R,M,R}\ci Tw_{M,R}=\\
&\eqby{SMC3}T\eps_M\ci Tw_{M,R}=\\
&\eqby{tet-counit spec}T\runi_M\ci T(M\ot T^0)\,.
\end{align*}
In order to prove (\ref{H0}) we need some preparation.
\begin{align}
\notag
&w_{M,N}\ci(M\ot\mu_N)=\\
\notag
&=w_{M,N}\ci(\runi_M\ot TN)\ci\asso_{M,R,TN}\ci(M\ot w^{-1}_{R,N})
\ci(M\ot(\eps_R\smp N))\ci(M\ot \gamma_{R,R,N})=\\
\notag
&=w_{M,N}\ci(\runi_M\ot TN)\ci((M\ot\eps_R)\ot  TN)\ci\asso_{M,TR,TN}\ci(M\ot w^{-1}_{TR,N})\ci(M\ot \gamma_{R,R,N})=\\
\notag
&\eqby{tet-counit spec}w_{M,N}\ci(\eps_M\ot TN)\ci(w_{M,R}\ot TN)\ci\asso_{M,TR,TN}\ci(M\ot w^{-1}_{TR,N})
\ci(M\ot \gamma_{R,R,N})=\\
\notag
&\eqby{P** spec}w_{M,N}\ci(\eps_M\ot TN)\ci w^{-1}_{M\smp R,N}\ci\gamma_{M,R,N}\ci w_{M,TN}=\\
\notag
&=(\eps_M\smp N)\ci\gamma_{M,R,N}\ci w_{M,TN}=\\
\label{mu-w}
&=\mu_{M,N}\ci w_{M,TN}
\end{align}
where in the first line we inserted an identity arrow in the form of the $\ot$-version of axiom (\ref{SMC4}), using also (\ref{w_R,N}),
and in the last line we used the notation of (\ref{mu_K,L}). It follows that
\begin{align}
\notag
(TM\ot\mu_N)\ci h_{M,TN}&=(TM\ot\mu_N)\ci w^{-1}_{TM,TN}\ci\gamma_{R,M,TN}\ci Tw_{M,TN}=\\
\notag
&\eqby{mu-w}w^{-1}_{TM,N}\ci\mu_{TM,N}\ci\gamma_{R,M,TN}\ci Tw_{M,TN}=\\
\notag
&\eqby{SMC1}w^{-1}_{TM,N}\ci(\eps_{TM}\smp N)\ci(\gamma_{R,M,R}\smp N)\ci\gamma_{R,M\smp R,N}\ci T\gamma_{M,R,N}
\ci Tw_{M,TN}=\\
\notag
&\eqby{SMC3}w^{-1}_{TM,N}\ci\gamma_{R,M,N}\ci T\mu_{M,N}\ci Tw_{M,TN}=\\
\notag
&\eqby{mu-w}w^{-1}_{TM,N}\ci\gamma_{R,M,N}\ci Tw_{M,N}\ci T(M\ot\mu_N)=\\ 
\label{mu-H}
&=h_{M,N}\ci T(M\ot\mu_N)
\end{align}
which is relation (\ref{H0}). While (\ref{H6}) obviously follows from (\ref{SMC5}) the proof of (\ref{H5}) needs some work:
\begin{align*}
&(\mu_M\ot TN)\ci h_{TM,N}\ci Th_{M,N}=\\
&=(\mu_M\ot TN)\ci w^{-1}_{T^2M,N}\ci\gamma_{R,TM,N}\ci Tw_{TM,N}\ci Tw^{-1}_{TM,N}\ci T\gamma_{R,M,N}\ci T^2w_{M,N}=\\
&=w^{-1}_{TM,N}\ci(\mu_M\smp N)\ci\gamma_{R,TM,N}\ci T\gamma_{R,M,N}\ci T^2w_{M,N}=\\
&\eqby{SMC1}w^{-1}_{TM,N}\ci((\eps_R\smp M)\smp N)\ci\gamma_{R\smp R,M,N}\ci\gamma_{R,R,M\smp N}\ci T^2w_{M,N}=\\
&=w^{-1}_{TM,N}\ci\gamma_{R,M,N}\ci\mu_{M\smp N}\ci T^2w_{M,N}=h_{M,N}\ci\mu_{M\ot TN}\,.
\end{align*}
\end{proof}

\begin{pro} \label{pro: T opmon}
Given a monoidal structure $\ot$ and a right-monoidal structure $\smp$ on the same category and with the same unit object $R$
the existence of a natural isomorphism $w_{M,N}:M\ot(R\smp N)\to M\smp N$ satisfying equations (\ref{P** spec}) and
(\ref{tet-counit spec}) implies that the formulas (\ref{T(w)}), (\ref{T0(w)}) define a $\ot$-opmonoidal structure for the
canonical monad $T=\bra R\smp\under,\mu,\eta\ket$ of the $\smp$-structure. 
\end{pro}
\begin{proof}
This is an immediate consequence of Lemma \ref{lem: H} and Proposition \ref{pro: fusion}.
\end{proof}

\begin{thm} \label{thm: representability}
Let $\bra \E,\ot, R,\asso,\luni^{-1},\runi\ket$ be a monoidal category. Then for a right-monoidal structure $\smp$ on $\E$ with unit object $R$ the following conditions are equivalent:
\begin{enumerate}
\item The $\smp$-structure is $\ot$-representable (by a $\ot$-bimonad) in the sense of Definition \ref{def: repr}.

\item There exists a natural isomorphism $w_{M,N}:M\ot (R\smp N)\to M\smp N$ satisfying the heptagon (\ref{P** spec}) and 
the tetragon (\ref{tet-counit spec}).

\item There exists a tetrahedral isomorphism $\tet_{L,M,N}:L\ot(M\smp N)\iso(L\ot M)\smp N$.
\end{enumerate}
\end{thm}
\begin{proof}
Equivalence of (ii) and (iii) has been shown in Proposition \ref{pro: tet-w}. Assume (i). This means that there exist a bimonad 
$\bra O,\omega,\iota\ket$ w.r.t. the $\ot$-structure and a skew-twist $v_{M,N}:M\odot N\to M\smp N$ where $\odot$ is
the skew-monoidal structure induced by $O$ in the sense of Proposition \ref{pro: O-ind}. Therefore $v$ satisfies the relations
\begin{align}
\label{twist1}
v_{L\smp M,N}\ci(v_{L,M}\ot ON)\ci\dot\gamma_{L,M,N}&=\gamma_{L,M,N}\ci v_{L,M\smp N}\ci(L\ot Ov_{M,N})\\
\label{twist2}
v_{R,N}\ci\dot\eta_N&=\eta_N\\
\label{twist3}
\dot\eps_M&=\eps_M\ci v_{M,R}
\end{align}
where $\dot\gamma$, $\dot\eta$, $\dot\eps$ are the expressions (\ref{dotgamma}), (\ref{doteta}), (\ref{doteps}).
We claim that the composite
\begin{equation}
w_{M,N}:=\left(M\ot TN \longrarr{M\ot v^{-1}_{R,N}}M\ot(R\odot N)\longrarr{M\ot\luni_{ON}}M\odot N
\longrarr{v_{M,N}}M\smp N\right)
\end{equation}
is a natural isomorphism satisfying (\ref{P** spec}) and (\ref{tet-counit spec}). 
With the notation $u_N:=\luni_{ON}\ci v^{-1}_{R,N}$ the left hand side of (\ref{P** spec}) can be transformed to the right hand side 
as follows.
\begin{align*}
&v_{L\smp M,N}\ci(v_{L,M}\ot ON)\ci((L\ot u_M)\ot ON)\ci ((L\ot TM)\ot u_N)\ci\asso_{L,TM,TN}\ci(L\ot (TM\ot u_N^{-1}))\\
&\qquad\ci(L\ot v^{-1}_{TM,N})\ci(L\ot\gamma_{R,M,N})=\\
&=v_{L\smp M,N}\ci(v_{L,M}\ot ON)\ci\asso_{L,OM,ON}\ci (L\ot(u_M\ot ON))\ci(L\ot v^{-1}_{TM,N})\ci(L\ot\gamma_{R,M,N})=\\
&\eqby{twist1}v_{L\smp M,N}\ci(v_{L,M}\ot ON)\ci\asso_{L,OM,ON}\ci(L\ot(\luni_{OM}\ot ON))\ci(L\ot\dot\gamma_{R,M,N})\\
&\qquad\ci(L\ot(R\ot Ov^{-1}_{M,N}))\ci(L\ot v^{-1}_{R,M\smp N})=\\
&\eqby{dotgamma}v_{L\smp M,N}\ci(v_{L,M}\ot ON)\ci\asso_{L,OM,ON}\ci(L\ot\luni_{OM\ot ON})\\
&\qquad\ci(L\ot(R\ot(OM\ot\omega_N)))\ci(L\ot(R\ot O^{M,ON}))\ci(L\ot(R\ot Ov^{-1}_{M,N}))\ci(L\ot v^{-1}_{R,M\smp N})=\\
&=v_{L\smp M,N}\ci(v_{L,M}\ot ON)\ci\asso_{L,OM,ON}\ci(L\ot(OM\ot\omega_N))\ci(L\ot O^{M,ON})\\
&\qquad\ci(L\ot Ov^{-1}_{M,N})\ci(L\ot u_{M\smp N})=\\
&=v_{L\smp M,N}\ci(v_{L,M}\ot ON)\ci\dot\gamma_{L,M,N}\ci(L\ot Ov^{-1}_{M,N})\ci(L\ot u_{M\smp N})=\\
&\eqby{twist1}\gamma_{L,M,N}\ci v_{L,M\smp N}\ci(L\ot u_{M\smp N})=\\
&=\gamma_{L,M,N}\ci w_{L,M\smp N}\,.
\end{align*}
In order to prove (\ref{tet-counit spec}) we compute its left hand side 
\begin{align*}
\eps_M\ci w_{M,R}&\eqby{twist3}\dot\eps_M\ci(M\ot u_R)=\\
&=\dot\eps_M\ci(\runi_M\ot OR)\ci\asso_{M,R,OR}\ci(M\ot v^{-1}_{R,R})=\\
&=\runi_M\ci\dot\eps_{M\ot R}\ci\asso_{M,R,OR}\ci(M\ot v^{-1}_{R,R})=\\
&\eqby{doteps}\runi_M\ci\runi_{M\ot R}\ci((M\ot R)\ot O^0) \ci\asso_{M,R,OR}\ci(M\ot v^{-1}_{R,R})=\\
&=\runi_M\ci(M\ot \runi_R)\ci(M\ot(R\ot O^0))\ci(M\ot v^{-1}_{R,R})=\\
&\eqby{doteps}\runi_M\ci(M\ot\dot\eps_R)\ci(M\ot v^{-1}_{R,R})=\\
&\eqby{twist3}\runi_M\ci(M\ot\eps_R)
\end{align*}
and arrive to to the expression on the right hand side. This proves the implication (i)$\Rightarrow$(ii).

Now assume (ii). Then we know by Proposition \ref{pro: T opmon} that $T$ is a bimonad, so by Proposition \ref{pro: O-ind}
that $M\odot N:=M\ot TN$ is a right-monoidal product. Therefore $\ot$-representability of the $\smp$-structure would follow
immediately if we could show that $w_{M,N}:M\odot N\to M\smp N$ is a twist.
\begin{align*}
&w_{L\smp M,N}\ci(w_{L,M}\ot TN)\ci\dot\gamma_{L,M,N}=\\
&=w_{L\smp M,N}\ci(w_{L,M}\ot TN)\ci\asso_{L,TM,TN}\ci(L\ot(TM\ot\mu_N))\ci(L\ot T^{M,TN})=\\
&\eqby{H(w)}(w_{L,M}\smp N)\ci w_{L\ot TM,N}\ci\asso_{L,TM,TN}\ci(L\ot w^{-1}_{TM,N})\ci(L\ot\gamma_{R,M,N})
\ci(L\ot Tw_{M,N})=\\
&\eqby{P** spec}\gamma_{L,M,N}\ci w_{L,M\smp N}\ci(L\ot Tw_{M,N})
\end{align*}
proves the hexagon relation  (\ref{twist1}) for $w$. The following simple computations yield the remaining relations:
\begin{align*}
w_{R,N}\ci\dot\eta&\eqby{doteta}w_{R,N}\ci\luni^{-1}_{TN}\ci\eta_N\eqby{w_R,N}\eta_N\\
\eps_M\ci w_{M,R}&=\runi_M\ci(M\ot\eps_R)=\runi_M\ci(M\ot T^0)\eqby{doteps}\dot\eps_M\,.
\end{align*}
So, $w$ is indeed a twist and this finishes the proof of the implication (ii)$\Rightarrow$(i).
\end{proof}

\section{Closed skew-monoidal categories} \label{s: closed rep}

A skew-monoidal category $\bra\M,\smp,R,\gamma,\eta,\eps\ket$ is called left (right) closed if the endofunctor $\under\smp N$
(resp. $N\smp\under$) has a right adjoint $\hom^l(N,\under)$ (resp. $\hom^r(N,\under)$) for all object $N\in\M$. It is called closed if it is both left closed and right closed.

\begin{thm} \label{thm: bgd}
Let $R$ be a ring. Then closed right-monoidal structures $\bra\Ab_R,\smp,R,\gamma,\eta,\eps\ket$ on the category of right $R$-modules, with unit object being the right-regular $R$-module, are precisely the right bialgebroids over $R$.
\end{thm}
\begin{proof}
In Section \ref{s: bgd} we have shown how bialgebroids over $R$ give rise to right-monoidal structures on $\Ab_R$.
The definition of the right-monoidal product (\ref{smp-bgd}) makes it obvious that it is closed.

Let $\smp$ be a closed right-monoidal structure on $\Ab_R$. 
Since $\Ab_R$ is cocomplete and $\under\smp N$ is left adjoint,  by the Eilenberg-Watts Theorem there is an isomorphism
\[
v_{M,N}:M\oR TN\iso M\smp N
\]
natural in $M$ for each $N$ where $\oR$ stands for the action on the monoidal category $_R\Ab_R$ on $\Ab_R$. 
(Note that the left $R$-module structure of $TN=R\smp N$ is defined by the endomorphism ring  of the right-regular module $R$,
i.e., by $\lambda_1$ in the notation of Section \ref{s: E-obj}.)
Without loss of generality we may assume that $v$ also satisfies the normalization
\begin{equation} \label{v_R,N}
v_{R,N}=\luni_{TN}
\end{equation}
for each $N$. (Otherwise compose it with $(M\ot (\luni_{TN}\ci v_{R,N}^{-1}))$.) Then considering $N\mapsto (\under\smp N)$ as the object map of a functor $\Ab_R\to\End\Ab_R$ the $v_{M,N}$ becomes natural in $N$, too.
Now substituting $v$ for $w$ in the heptagon (\ref{P** spec}) with $L=R$ we obtain an identity due to (\ref{v_R,N}). 
Similarly, (\ref{tet-counit spec}) with $w=v$ and $M=R$ is an identity. Therefore, using that $R$ is a generator, it follows
that both  (\ref{P** spec}) and (\ref{tet-counit spec}) are identities for all values of their arguments $L$, $M$ and $N$.

Next we want to construct a $w$ for the quotient right-monoidal structure $\smpq$ (see Proposition \ref{pro: gamma+}) on the monoidal category $_R\Ab_R$. There is a unique $w$ such that for all $M,N\in\,_R\Ab_R$
\begin{equation}
\begin{CD}
M\oR TN@>v_{M,N}>> M\smp N\\
@V{M\oR q_{R,N}}VV @VV{q_{M,N}}V\\
M\oR T_qN@>w_{M,N}>>M\smpq N
\end{CD}
\end{equation}
since $q_{M,N}$ is a coequalizer. $w_{M,N}$ is invertible since $M\oR\under$ preserves coequalizers.
Now use (\ref{q-hexa}), (\ref{q-squares}) to show that the heptagon  (\ref{P** spec}) and tetragon (\ref{tet-counit spec}) for $v$ and $\smp$ implies the heptagon and tetragon for $w$ and $\smpq$. Then by Theorem \ref{thm: representability} $T_q$ is
a bimonad on $_R\Ab_R$. Thus we could conclude by \cite[Theorem 4.5]{Sz: EM} that
$T_q$ is the bimonad of a bialgebroid if we knew that $T_q$ is left adjoint. Using that $\smp$ is also right closed the
Eilenberg-Watts Theorem provides an isomorphism $M\smp N\cong N\am{R_2}(M\smp R)$; hence $TN\cong N\am{R_2}H$
where $H=R\smp R$. The quotient
\[
T_qN=\ \int^{\rho_1\lambda_2}TN\ \cong\ \int^{\rho_1\lambda_N}(N\am{R_2}H)\ \cong\ N\am{R^e}H
\]
amalgamates the left $R$-action on $N$ with the right $R$-action $\rho_1$ on $H$ which, together with $\am{R_2}$, amounts
to taking tensor product over $R^e=R^\op\ot R$ by considering $N$ as right $R^e$-module and $H$ as left $R^e$-module
via $(r'\ot r)\cdot h=\rho_1(r')\ci\lambda_2(r) (h)$. As such, $T_q$ is left adjoint.
\end{proof}

Combining the above result with Mitchell's Theorem on the characterization of module categories we can obtain
a characterization of skew-monoidal categories of bialgebroids without explicit reference to the base ring.
\begin{cor}
A right monoidal category $\bra\M,\smp,R,\gamma,\eta,\eps\ket$ is equivalent to the right-monoidal category of a right-bialgebroid iff
\begin{enumerate}
\item $\M$ is cocomplete abelian,
\item $\smp$ preserves colimits in both arguments
\item and $R$ is a small projective generator.
\end{enumerate}
\end{cor}

\section{Monoidal (lax) comonads}

In this last section, we discuss two results that lead to monoidality of the canonical lax comonad of a skew-monoidal category.

\subsection{The corepresentability theorem}

We would like to characterize the skew-monoidal categories that can be ``corepresented'' in the sense of Definition \ref{def: corep}
by a monoidal comonad. For that purpose we dualize the construction of Section \ref{s: rep}.

Let $\bra\E,\smp,R\ket$ be a right-monoidal category the dual $\bra\E^\op,\smp^\op,R\ket$ of which is representable by an opmonoidal monad in the right-monoidal category $\bra\E^\op,\ot,R\ket$. This means precisely that the original $\smp$-structure
is corepresentable by a monoidal comonad w.r.t the left-monoidal structure $\ot$. So we can speak about tetrahedral homomorphisms $t$ as natural transformations
\[
t_{L,M,N}\ :\ N\smp(L\ot M)\ \to\ L\ot(N\smp M)
\]
satisfying the pentagons
\begin{align*}
(K\ot t_{L,M,N})\ci t_{K,L\ot M,N}\ci(N\smp\asso^{-1}_{K,L,M})&=\asso^{-1}_{K,L,N\smp M}\ci t_{K\ot L,M,N}\\
(K\ot\gamma_{N,M,L})\ci t_{K,M\smp L,N}\ci(N\smp t_{K,L,M})&=t_{K,L,N\smp M}\ci\gamma_{N,M,K\ot L}
\end{align*}
and the triangles
\begin{align*}
\luni_{N\smp M}\ci t_{R,M,N}&=N\smp \luni_M\\
t_{M,N,R}\ci\eta_{M\ot N}&=M\ot\eta_N\,.
\end{align*}
(We have written $t$ exactly for what it was in Section \ref{s: rep}, without even permuting indices, now using the opposite
composition and opposite skew-monoidal product.) Such a $t$ is then a tetrahedral isomorphism if
\[
w_{M,N}:=t_{M,R,N}\ci(N\smp\runi^{-1}_M)\ :\ N\smp M\ \to\ M\ot QN
\]
is a natural isomorphism. 

Dualizing Proposition \ref{pro: tet-w} we obtain that $t$ is a tetrahedral isomorphism if and only if
$w$ satisfies the following heptagon and tetragon equations:
\begin{align}
\label{coheptagon}
w_{L,N\smp M}\ci\gamma_{N,M,L}&=(L\ot\gamma_{N,M,R})\ci(L\ot w^{-1}_{QM,N})\ci\asso^{-1}_{L,QM,QN}\ci 
w_{L\ot QM,N}\ci(N\smp w_{L,M})\\
\label{cotetragon}
w_{M,R}\ci\eta_M&=(M\ot \eta_R)\ci\runi^{-1}_M\,.
\end{align}

The fusion operators can be defined as the composite natural transformation
\[
h_{M,N}:=Qw_{M,N}\ci\gamma_{N,M,R}\ci w^{-1}_{QM,N}\ :\ QM\ot QN\ \to\ Q(M\ot QN)\,.
\]
This allows to write up the would-be monoidal structure for the canonical comonad $Q=\bra \under\smp R,\delta,\eps\ket$
as follows
\begin{align}
Q_{M,N}&:=Q(M\ot\eps_N)\ci h_{M,N}\ :\ QM\ot QN\ \to\ Q(M\ot N)\\
Q_0&:=\eta_R\ :\ R\ \to\ QR\,.
\end{align}
Then by dualizing Theorem \ref{thm: representability} we obtain the following corepresentability theorem:
\begin{thm} \label{thm: corepresentability}
Let $\E$ be a category equipped with a right-monoidal structure $\smp$ and a monoidal structure $\ot$ with a common unit object $R$. Then the following statements are equivalent.
\begin{enumerate}
\item $\smp$ is $\ot$-corepresentable, i.e., there is a $\ot$-monoidal comonad $C$ and a twist-isomorphism 
$M\smp N\iso N\ot CM$ of right monoidal structures.
\item There is a natural isomorphism $w_{M,N}:N\smp M\iso M\ot QN$ satisfying the heptagon and tetragon equations (\ref{coheptagon}) and (\ref{cotetragon}) where $Q$ is the canonical comonad of the $\smp$-structure.
\item There is a tetrahedral isomorphism $t_{L,M,N}:N\smp(L\ot M)\to L\ot(N\smp M)$. 
\end{enumerate}
\end{thm}

One may try to apply this corepresentation theorem to a situation dual to that of Section \ref{s: closed rep}, e.g., by considering
categories of right comodules of a coalgebra and coclosed skew-monoidal structures on them. 
Unfortunately this dualization seems to require more than what is 
known, to the present author, about bicoalgebroids \cite{Brz-Mil, Balint}.

\subsection{Monoidality of the lax comonad on $_R\sfM_R$}

If $\E$ is a monoidal category then monoidality of the lax comonad $\bfQ:\Delta^\op\to \End \E$ means the structure on $\bfQ$ that allows its factorization through the 
faithful functor $\End^\ot\E\into\End\E$ which forgets monoidality of monoidal endofunctors and their monoidal natural transformations. 
If $\E$ is the category of $E$-objects of a complete right-monoidal category $\M$ and $\bfQ$ is the lax comonad on $\E$ constructed in Section \ref{s: lax Q} then one would like to find conditions on a monoidal structure $\ot$ on $\E$ which implies monoidality of $\bfQ$. For the monad $T_q$ the existence of tetrahedral isomorphism between $\ot$ and $\smpq$ on $\E$ implied its opmonoidality. Unfortunately we do not know analogous conditions that would imply monoidality of $\bfQ$.
However, if $\E$ is the category $_R\Ab_R$ of bimodules over a ring $R$ and $\bfQ$ is the lax comonad of a right $R$-bialgebroid
one expects that monoidality of $\bfQ$ follows without any additional conditions. 

As the proof of \cite[Proposition 4.2]{Day-Street} indicates, in order to construct the monoidal structure of $\bfQ$, it is not sufficient to work within $_R\Ab_R$, it has to be embedded into a monoidal bicategory of bimodules. The basic idea of the 
proof of the next Theorem is that of the above mentioned construction of \cite{Day-Street} although some differences in the conventions may disguise it. 
\begin{thm} \label{thm: bfQ monoidal}
For a commutative ring $k$ and a $k$-algebra $R$ let $\bra\sfM_R,\smp,R,\gamma,\eta,\eps\ket$ be a closed right-monoidal
structure on the category of right $R$-modules. Then the lax comonad $\bfQ$ on $_R\sfM_R$ defined in Proposition \ref{pro: bfQ} is monoidal and the Eilenberg-Moore category $_R\sfM_R^\bfQ$ has a unique monoidal structure such that the forgetful functor
$_R\sfM_R^\bfQ\to\,_R\sfM_R$ is strict monoidal.
\end{thm}
\begin{proof}
Let $\E(m,n)$ be the category of $R_n$-$R_m$-bimodules where $R_n:=R\ot (R^\op\ot R)^{\ot(n-1)}$ and $\ot$ denotes tensor product over $k$. Tensor product over $R_n$ is denoted by $\bo$ for any $n$. 

Let $H$ denote $R\smp R$ as an $R^\op\ot R$-bimodule. 
Since $H$ is a monoid in the category of $R^\op\ot R$-bimodules, tensoring with $H$ ($n$ times) defines monoidal functors
$\ch^n:\E(1,l)\to\E(n+1,n+l)$ given recursively by $\ch^0M:=M$ and $\ch^nM:=\ch^{n-1}M\ot H$ if $n>0$.

Let $P\in\E(1,2)$ be the $k$-module $R\ot R$ equipped with $(R\ot R^\op\ot R)$-$R$-bimodule structure
\[
(r_1\ot r' \ot r_2)\cdot(x\ot y)\cdot r_3:=r_1 xr'\ot r_2yr_3\,.
\]
We shall also need the $n$-th iterate of $P$
\[
P_1:=P\quad\text{and}\quad P_n:=(P\ot R_{n-1})\bo P_{n-1}\in\E(1,n+1), \qquad n>1\,.
\]

Since $\under\smp N$ is left adjoint for each $N\in\sfM_R$, there is an isomorphism $M\smp N\iso M\oR (R\smp N)$, natural in $M$, where the left $R$-module structure of $R\smp N$ is given by $\lambda_1$. Setting $N=R$ we obtain 
$QM\iso M\oR H=\ch M\bo P$ and iterating $Q^nM\iso \ch^nM\bo P_n$. 

Using that $P_n\bo\under: \E(1,1)\to\E(1,n+1)$ has a right adjoint the object map of the lax comonad $\bfQ$ can be given by the
functors 
\[
M\mapsto \bfQ_nM =\Hom_{R_{n+1}}(P_n,Q^nM)\iso \Hom_{R_{n+1}}(P_n,\ch^n M\bo P_n)
\]
The counit of this hom-tensor adjunction, i.e., the evaluation $\ev^n:P_n\bo\Hom_{R_{n+1}}(P_n,\under)\to \under$, allows us to define $(\bfQ_n)_{M,N}$ by the following commutative diagram (in which the associators for $\bo$ are
suppressed and $\ev^n_M$ is written instead of $\ev^n_{H^nM\bo P_n}$ for brevity)
\begin{equation} \label{eq: multip of bfQ}
\begin{CD}
P_n\bo\bfQ_nM\bo\bfQ_nN@>(1\bo\ev^n_N)\ci(\ev^n_M\bo 1)>>\ch^nM\bo\ch^nN\bo P_n\\
@V{1\bo(\bfQ_n)_{M,N}}VV @VV{(\ch^n)_{M,N}\bo 1}V\\
P_n\bo\bfQ_n(M\bo N) @>\ev^n_{M\bo N}>>\ch^n(M\bo N)\bo P_n
\end{CD}
\end{equation}
The unit $(\bfQ_n)_0:R\to \bfQ_n R$, in turn, is defined by the unit of $\ch^n$ via the diagram
\begin{equation} \label{eq: unit of bfQ}
\begin{CD}
P_n\bo R@>\sim>>R_{n+1}\bo P_n\\
@V{1\bo (\bfQ_n)_0}VV @VV{(\ch^n)_0\bo 1}V\\
P_n\bo\bfQ_n R@>\ev^n_R>>\ch^nR\bo P_n
\end{CD}
\end{equation}
That $(\bfQ_n)_{M,N}$ and $(\bfQ_n)_0$ make $\bfQ_n$ a monoidal functor is now a simple consequence of monoidality of the 
functor $\ch^n$.

Next we have to show that $\bfQ_f$ is a monoidal natural transformation for all $f:m\to n$ in $\Delta$.
For $f=i+(2\to 1)+(n-1-i)$ this means showing commutativity of the diagrams
\[
\begin{CD}
\bfQ_nM\bo\bfQ_nN@>(\bfQ_n)_{M,N}>>\bfQ_n(M\bo N)\\
@V{\delta_n^iM\bo\delta_n^iN}VV @VV{\delta_n^i(M\bo N)}V\\
\bfQ_{n+1}M\bo\bfQ_{n+1}N@>(\bfQ_{n+1})_{M,N}>>\bfQ_{n+1}(M\bo N)
\end{CD}
\qquad\qquad
\begin{CD}
R@>(\bfQ_n)_0>>\bfQ_n R\\
@| @VV{\delta_n^i}V\\
R@>(\bfQ_{n+1})_0>>\bfQ_{n+1}R
\end{CD}
\]
To make a long story short, we already know by Theorem \ref{thm: bgd} that $H$ is a right $R$-bialgebroid therefore the factorization of the comultiplication $\Delta^\sH:H\rarr{\Delta^\sH_\x}H\x H\into H\oR H$ through the Takeuchi product is an algebra map $\Delta^\sH_\x$. Commutativity of the above two diagrams follows precisely from multiplicativity and unitality of $\Delta^\sH_\x$. Similar observation for the counit leads to monoidality of $\eps_n^i$. This defines the required factorization of
the functor $\bfQ:\Delta^\op\to\End\E(1,1)$ through the category $\End^\ot \E(1,1)$ of monoidal endofunctors and monoidal natural transformations.

It remains to show that the monoidal structure of $\bfQ$, namely $\nu$ and $\iota$, consists also of monoidal natural transformations. For $\iota$ there is nothing to prove since it can be chosen to be the identity as we have seen in the proof of Proposition \ref{pro: bfQ}. For $\nu$ this is the commutativity of the diagrams
\begin{equation} \label{nu monoidal}
\begin{CD}
\bfQ_m\bfQ_nM\bo\bfQ_m\bfQ_nN@>>>\bfQ_m\bfQ_n(M\bo N)\\
@V{\nu^{m,n}M\bo\nu^{m,n}N}VV @VV{\nu^{m,n}(M\bo N)}V\\
\bfQ_{m+n}M\bo\bfQ_{m+n}N@>>>\bfQ_{m+n}(M\bo N)
\end{CD}
\qquad\qquad
\begin{CD}
R@>>>\bfQ_m\bfQ_n R\\
@| @VV{\nu^{m,n}R}V\\
R@>>>\bfQ_{m+n}R
\end{CD}
\end{equation}
Since $P_{m+n}=(P_n\ot(R^\op\ot R)^m)\bo P_m$, we obtain the following multiplicativity rule for the evaluation:
\begin{equation} \label{eq: multip of ev}
\begin{CD}
P_{m+n}\bo\bfQ_m\bfQ_n M@>1\bo\ev^m_{\bfQ_nM}>>H^m(P_n\bo\bfQ_n M)\bo P_m\\
@V{1\bo\nu^{m,n}_M}VV @VV{H^m\ev^n_M\bo 1}V\\
P_{m+n}\bo\bfQ_{m+n}M@>\ev^{m+n}_M>>H^{m+n}M\bo P_{m+n}
\end{CD}
\end{equation}
Using (\ref{eq: multip of bfQ}) and (\ref{eq: multip of ev}) one can show that
\begin{gather*}
\ev^{m+n}_{M\bo N}\ci(P_{m+n}\bo\nu^{m,n}_{M\bo N})\ci(P_{m+n}\bo(\bfQ_m\bfQ_n)_{M,N})=\\
=\ev^{m+n}_{M\bo N}\ci(P_{m+n}\bo(\bfQ_{m+n})_{M,N})\ci(P_{m+n}\bo\nu^{m,n}_M\bo\nu^{m,n}_N)
\end{gather*}
from which the first diagram in (\ref{nu monoidal}) follows by adjunction. As for the second diagram one utilizes the fact that $\ch^{m+n}=\ch^m\ch^n$ in diagram (\ref{eq: unit of bfQ}) to obtain 
\[
\ev^{m+n}_R\ci(1\bo(\bfQ_{m+n})_0)=(H^m\ev^n_R\bo 1)\ci (H^m(1\bo(\bfQ_n)_0)\bo 1)\ci(1\bo\ev^m_R)\ci(1\bo 1\bo(\bfQ_m)_0)
\]
from which the statement can be obtained by rewriting the RHS using (\ref{eq: multip of ev}). This finishes the proof
of monoidality of the lax comonad. The way the Eilenberg-Moore forgetful functor becomes strict monoidal is standard and 
needs no explanation.
\end{proof}

\end{document}